\documentclass[11pt]{amsart}
\usepackage[utf8]{inputenc}

\usepackage{amsmath,amsfonts,amsthm, amssymb}
\usepackage[hidelinks]{hyperref}
\usepackage[abbrev, msc-links, nobysame]{amsrefs}
\usepackage{fullpage}

\numberwithin{equation}{section}

\theoremstyle{plain}
\newtheorem{lemma}{Lemma}[section]
\newtheorem{theorem}[lemma]{Theorem}
\newtheorem{prop}[lemma]{Proposition}

\newtheorem{problem}[lemma]{Problem}

\theoremstyle{definition}

\newtheorem{defn}[lemma]{Definition}

\theoremstyle{remark}
\newtheorem{remark}[lemma]{Remark}
\newtheorem*{remark*}{Remark}

\title{Banach's isometric subspace problem in dimension four}

\author{Sergei Ivanov}
\author{Daniil Mamaev}
\author{Anya Nordskova}

\address{Sergei Ivanov: St.Petersburg Department of Steklov Mathematical Institute,
Fontanka 27, St.Petersburg 191023, Russia}
\email{svivanov@pdmi.ras.ru}

\address{Daniil Mamaev: St.Petersburg Department of Steklov Mathematical Institute,
Fontanka 27, St.Petersburg 191023, Russia}
\email{dan.mamaev@gmail.com}

\address{Anya Nordskova: St.Petersburg Department of Steklov Mathematical Institute,
Fontanka 27, St.Petersburg 191023, Russia, and Universiteit Hasselt, Agoralaan, gebouw D, 3590 Diepenbeek, Belgium}
\email{anya.nordskova@uhasselt.be}



\newcommand{\R}{\mathbb R}

\newcommand{\I}{\mathcal I}

\renewcommand{\sl}{\mathrm{sl}}
\newcommand{\vol}{\mathrm{vol}}
\newcommand{\Gr}{\mathrm{Gr}}
\newcommand{\id}{\mathrm{id}}

\DeclareMathOperator{\tr}{Trace}

\DeclareMathOperator{\dist}{dist}
\DeclareMathOperator{\linspan}{LinSpan}

\DeclareMathOperator{\GL}{GL}

\DeclareMathOperator{\pr}{pr}

\DeclareMathOperator{\Hom}{Hom}
\newcommand{\la}{\lambda}
\newcommand{\pd}{\partial}
\newcommand{\co}{\colon}
\newcommand{\ep}{\varepsilon}
\newcommand{\ga}{\gamma}

\begin{document}

\begin{abstract}
We prove that if all intersections of a convex body $B\subset\mathbb R^4$ 
with 3-dimensional linear subspaces are linearly equivalent 
then $B$ is a centered ellipsoid.
This gives an affirmative answer to the case $n=3$ of the following question by Banach from 1932: Is a normed vector space $V$ whose $n$-dimensional linear subspaces are all isometric, for a fixed $2 \le n< \dim V$, necessarily Euclidean?

The dimensions $n=3$ and $\dim V=4$ is the first case where the question was unresolved. Since the $3$-sphere is parallelizable, known global topological methods do not help in this case. Our proof employs a differential geometric approach.
\end{abstract}

\maketitle

\section{Introduction}
 
In this paper we give an affirmative answer to the case $n=3$ of
the following problem:

\begin{problem}\label{problem with norm}
Let $(V,\|\cdot\|)$ be a normed vector space (over $\R$) such that
for some fixed $n$, ${2\le n<\dim V}$, all $n$-dimensional
linear subspaces of $V$ are isometric.
Is $\|\cdot\|$ necessarily a Euclidean norm
(i.e., an inner product one)?
\end{problem}

The problem goes back to Banach 
\cite{Banach32}*{Remarks on Chapter XII} and is also known as ``Banach's conjecture'':
it is conjectured that the answer is always affirmative.
This has been proven for many but not all dimensions.
Auerbach, Mazur and Ulam \cite{AMU} proved the conjecture in the case $n=2$.
Dvoretzky \cite{Dvo} proved it for an infinite-dimensional $V$ and any $n \ge 2$
as a corollary of his famous theorem on approximately Euclidean subspaces.
Gromov \cite{Gro67} proved that the answer is affirmative
if $n$ is even or $\dim V\ge n+2$.
Recently Bor, Hern\'andez-Lamoneda, Jim\'enez-Desantiago and Montejano
\cite{BHJM21} extended Gromov's approach 
and proved the conjecture for $n$ congruent to $1$ modulo $4$ except $n=133$.
The problem also makes sense for complex normed spaces,
see \cites{Mil,Gro67,BM21} for some results in this direction.

Recall that (due to the parallelogram law)
a normed vector space $V$ is Euclidean if and only if
all its $2$-dimensional linear subspaces are.
This fact reduces Problem~\ref{problem with norm} to the case $\dim V = n + 1$,
and below we focus only on this case.  

We say that two sets $A, B\subseteq V$ 
are {\em linearly equivalent} if there is
a linear bijection $L\colon\linspan(A) \to \linspan(B)$ 
such that $L(A)=B$.
Considering the unit ball of the norm, one can reformulate 
Problem~\ref{problem with norm} in geometric terms as follows.

\begin{problem}\label{problem with body}
Let $n\ge 2$ and let $B\subset\R^{n+1}$
be a convex body 
symmetric with respect to the origin.
Assume that all intersections of $B$ with $n$-dimensional linear subspaces
are linearly equivalent.
Is $B$ necessarily an ellipsoid?
\end{problem}

The symmetry assumption for $B$ can be removed
due to to the following result of Montejano
\cite{Mon91}*{Theorem~1}:
If $B\subset\R^{n+1}$ is a convex body and
all intersections of $B$ with $n$-dimensional linear subspaces
are affine equivalent, then either $B$ is symmetric with respect to~0,
or $B$ is a (not necessarily centered) ellipsoid.
In particular, linear equivalence of the intersections implies that $B$ is symmetric with respect to~0.
Keeping this fact in mind, we do not assume the symmetry in the sequel
as some of our intermediate results may be of interest in the non-symmetric case. 

Our main result is the following theorem, which implies affirmative answers to the above problems for $n=3$.

\begin{theorem}\label{t:main}
Let $B\subset\R^4$ be a convex body such that
all intersections of~$B$ with 3-dimensional linear subspaces of $\R^4$ are linearly equivalent.
Then $B$ is an ellipsoid centered at~$0$.
\end{theorem}

In contrast with the aforementioned works \cites{AMU,Gro67,BHJM21}, the proof of Theorem~\ref{t:main} is essentially local: for the most part it deals with an
arbitrarily small neighborhood of one cross-section.
A similar approach was developed in \cite{mono} where a ``local''
variant of Banach's problem was solved for $n=2$ and
smooth strictly convex norms.
An analogous generalization is possible for Theorem \ref{t:main}
but it is outside the scope of this paper.
The present paper borrows many ideas from \cite{mono},
in particular Proposition \ref{prop:Rexists} is a 
generalization of constructions in \cite{mono}*{\S3}.
We believe that this approach to Banach's problem
is viable in all dimensions, but so far we were able
to make it work only for $n=2$ and $n=3$.

\subsection*{``Integrable'' and ``non-integrable'' problems}
The proofs in \cites{AMU, Gro67, BHJM21}
are based on topological methods.
In fact, the arguments in these proofs also solve the following ``non-integrable''
version of the problem where one considers a family of $n$-dimensional convex bodies
in linear subspaces
without requiring that they are cross-sections of a single body in $V$:

\smallskip

{\it
Let $V$ be a vector space, $2\le n<\dim V$, and
let $K\subset\R^n$ be a convex body.
Assume there exists a continuous family $\{K_X\}$ where
$X$ ranges over all $n$-dimensional linear subspaces of $V$
and each $K_X$ is a convex body in $X$ linearly equivalent to~$K$.
For what types of $K$ is this possible?
}

\smallskip

This problem is topological by its nature and it can be reduced
to the study of structure groups of certain bundles, see \cites{Gro67,BHJM21}. 
Gromov \cite{Gro67} showed that if $n=2k$ or $\dim V\ge n+2$
then $K$ in this problem must be an ellipsoid.
This obviously implies the respective cases of Banach's conjecture.
Although the ``non-integrable'' generalization fails for $n=2k+1=\dim V-1$,
in some cases one can obtain strong restrictions
on the geometry of the body in question.
Namely, in \cite{BHJM21} it is shown that
for $n=4k+1$, except $n=133$, $K$ must be an affine body of revolution
(i.e., has an $SO(n-1)$ symmetry), and
then it follows that in the original ``integrable'' setting $K$ is an ellipsoid.

We note that considering the non-integrable problem does not help if $n=3$ and $\dim V=4$
because in this case any centrally symmetric convex body can be chosen for~$K$.
This follows from the fact that the 3-sphere is parallelizable,
see \cite{Burton} for a detailed explanation.

\subsection*{Structure of the proof and organization of the paper}
The plan of the proof of Theorem~\ref{t:main} is detailed in Section \ref{sec:plan},
here is an overview.

The proof has two main steps.
In the first step, stated as Proposition \ref{prop:Rexists},
we obtain differential geometric implications of ``local integrability''
near a generic cross-section
$K=X\cap B$ where $X$ is an $n$-dimensional linear subspace.
Loosely speaking, Proposition \ref{prop:Rexists} provides us with
an algebraic $n(n-1)$-parameter family of vectors tangent
to the boundary of~$B$ at points of~$K$.
Furthermore it includes an $(n-1)^2$-parameter sub-family
of vectors tangent to the cross-section~$K$ itself,
see \eqref{e:tangency assumption}.

In the second step we study restrictions on the geometry of the cross-section~$K$
imposed by the existence of such a family of tangent vectors.
The resulting statement is Proposition \ref{prop:Rtangent}.
It asserts that, at least for $n=3$, the linear vector fields provided
by Proposition \ref{prop:Rexists}
are tangent to~$\pd K$ everywhere rather than only at certain planes
as in~\eqref{e:tangency assumption}.
The proof of Proposition \ref{prop:Rtangent} is the most substantial part of the paper, see \S\ref{subsec:plan main} for an outline.

Theorem \ref{t:main} follows from Propositions \ref{prop:Rexists}
and \ref{prop:Rtangent} with an application of
Blaschke-Kakutani ellipsoid characterization,
see \S\ref{subsec:proof main}.
The specifics of the dimension $n = 3$ are used only in the proof of Proposition~\ref{prop:Rtangent},
while the proof of Proposition~\ref{prop:Rexists} and the deduction of Theorem~\ref{t:main} from the propositions work for all $n \ge 2$.

The rest of the paper is organized as follows.
In Section \ref{sec:plan} we
introduce definitions and notation,
formulate Propositions \ref{prop:Rexists} and \ref{prop:Rtangent},
and deduce Theorem \ref{t:main} from them.
Section~\ref{sec: polynomial vector fields} is a collection of technical lemmas
(about polynomial vector fields tangent to convex hypersurfaces) used throughout the paper.
The remaining sections are devoted to the proofs of
Propositions \ref{prop:Rexists} and \ref{prop:Rtangent}.
The proof of Proposition~\ref{prop:Rexists} is given in Section~\ref{sec:proofRexists}.
In Section~\ref{sec:degenerate} we consider degenerate cases
of Proposition~\ref{prop:Rtangent},
and in Section~\ref{sec:non-degenerate} we handle the non-degenerate case
and finish the proof of Proposition~\ref{prop:Rtangent}.

\subsection*{Acknowledgement}
The authors are grateful to the anonymous referee for their
thorough reading of the manuscript and helpful suggestions.

This work is supported by the Russian Science Foundation under Grant 21-11-00040.

\section{Preliminaries and the plan of the proof}
\label{sec:plan}

In this section we set up terminology and notation,
formulate the two main propositions and deduce Theorem \ref{t:main} from them.

\subsection{Notation and preliminaries}
In this paper, a ``vector space''
always means a finite-dimensional real vector space.
For vector spaces $X$ and $Y$, $\Hom(X,Y)$ denotes
the space of linear maps from $X$ to~$Y$.
We denote $\Hom(X,\R)$ by $X^*$.

For  a vector space $X$, we denote by $\Gr_k(X)$ the Grassmannian manifold
of (unoriented) $k$-dimensional linear subspaces of~$X$ and
by $GL(X)$ the group of linear bijections from $X$ to itself.

For a convex set $K\subset X$, we denote by $\pd K$
the relative boundary of $K$, that is the boundary in the topology of its affine span,
and by $\I(K)$ the \emph{group of linear self-equivalences of $K$}:
\begin{equation}\label{e:defI(K)}
\I(K) = \{ L\in GL(X) : L(K)=K \} .
\end{equation}
A convex set $K \subset X$ is called a {\em convex body} if it is compact and
its interior is non-empty.
If $K$ is a convex body then $\I(K)$ is a compact subgroup of $GL(X)$.

By an {\em ellipsoid} we mean the unit ball of an inner product norm in a vector space.
In other words, all ellipsoids are assumed to be centered at~$0$ (unless explicitly stated otherwise). 
The same terminology adjustment applies to ellipses in dimension~2.

By a {\em Minkowski norm} on a vector space $V$ we mean a
function $\Psi\co V\to\R_+$ which is positively 1-homogeneous,
subadditive (and hence convex), and positive on $V\setminus\{0\}$.
The difference from usual norms is that a Minkowski norm
is not assumed symmetric.
There is a standard 1-to-1 correspondence between Minkowski norms on $V$ and
convex bodies in $V$ with 0 in the interior.
Namely, to each Minkowski norm $\Psi$ one associates its unit ball
$B_\Psi = \{ x\in V : \Psi(x)\le 1\}$, and for every convex
body $B\subset V$ with 0 in the interior there 
is a corresponding Minkowski norm $\Psi^B(x) = \inf\{\lambda > 0 \mid x/\lambda \in B\}$.

The convexity of $\Psi$ implies (see e.g.\ \cite{Clarke}*{2.1.1 and 2.2.7}) that
the {\em one-sided directional derivative}
$$
\partial_x^+\Psi(v) = \lim_{s \downarrow 0} \frac{\Psi(x + sv) - \Psi(x)}{s} 
$$ 
is well-defined for all $x, v \in V$. Moreover, $\partial_x^+\Psi \colon V \to \R$ is positively 1-homogeneous and subadditive for all $x \in V$. 

There are many equivalent definitions of one-sided tangent directions of a convex hypersurface. In the present paper we use the one by Bouligand:

\begin{defn} \label{d:forward tangent}
Let $V$ be a vector space, $B\subset V$ a convex body, $x\in\pd B$, and $v\in V$.
We say that $v$ is \textit{forward tangent to $\pd B$ at $x$}
if there exist sequences $\{x_i\}_{i=1}^\infty\subset\pd B$
and $\{c_i\}_{i=1}^\infty\subset\R_+$ such that $x_i\to x$
and $c_i(x_i-x)\to v$ as $i\to\infty$.

We say that $v$ is \textit{tangent to $B$ at $x$}
if  both $v$ and $-v$ are forward tangent to $B$ at~$x$.
\end{defn}

The following two lemmas are standard. 
We include their proofs for the reader's convenience
and to avoid a discussion on various definitions of tangency.

\begin{lemma}\label{l:tangent cone}
Let $V$ be a vector space and $B\subset V$
a convex body with 0 in the interior.
Let $x\in \pd B$ and $v\in V$. Then the following conditions
are equivalent:
\begin{enumerate}
\item $v$ is forward tangent to $\pd B$ at $x$.
\item $\pd^+_x\Psi(v) = 0$ where $\Psi$ is the Minkowski norm associated to~$B$.
\end{enumerate}
\end{lemma}

\begin{proof}
Both conditions are obviously satisfied for $v=0$, so we assume that $v\ne 0$.
Let $v$ be forward tangent to $\pd B$ at $x$ and let $\{x_i\}_{i=1}^\infty\subset\pd B$, $\{c_i\}_{i=1}^\infty\subset\R_+$ be the corresponding sequences from Definition~\ref{d:forward tangent}.
Since $v\ne 0$, we have $c_i\to +\infty$, hence may assume that $\{c_i\}_{i=1}^\infty$
is an increasing sequence. We have 
$$
\pd^+_x\Psi(v) = \lim_{i \to \infty} \frac{\Psi(x + v/c_i) - \Psi(x)}{1/c_i} = \lim_{i \to \infty} \Psi(c_ix + v) - c_i\Psi(x) = \lim_{i \to \infty} \Psi(c_ix + v) - \Psi(c_ix_i).
$$ 
By the sub-additivity of $\Psi$,
$$
-\Psi(c_i(x_i - x) - v) \le \Psi(c_ix + v) - \Psi(c_ix_i) \le \Psi(c_i(x - x_i) + v) .
$$
Since $c_i(x_i - x) \to v$ as $i\to\infty$, we conclude that $\pd^+_x\Psi(v) = 0$.

On the other hand, if $\pd^+_x\Psi(v) = 0$ then we can choose
any sequence $\{s_i\}_{i = 1}^\infty \subset \R_{+}$ decreasing to~$0$
and set $x_i$ and $c_i$ required in Definition~\ref{d:forward tangent}
to be $x_i = \frac{x + s_i v}{\Psi(x + s_i v)}$ and $c_i = \frac{\Psi(x + s_iv)}{s_i}$.
\end{proof}

\begin{lemma}\label{l:tangent linear}
Let $V$ be a vector space, $B\subset V$
a convex body with 0 in the interior, and $x\in \pd B$.
Then the set of all vectors tangent to $\pd B$ at $x$
is a linear subspace of~$V$.
\end{lemma}

\begin{proof}
By Lemma~\ref{l:tangent cone}, $v \in V$ is tangent to $\pd B$ at $x$ if and only if 
$\pd^+_x\Psi(v) = \pd^+_x\Psi(-v) = 0$.
Since $\pd^+_x\Psi$ is positively 1-homogeneous, this implies that
$\pd^+_x\Psi(tv) = 0$ for all $t\in \R$.

Let $v,w$ be tangent to $\pd B$ at $x$. Then, by the sub-additivity of $\pd^+_x\Psi$, 
$$
\pd^+_x\Psi(v+w)\le \pd^+_x\Psi(v)+\pd^+_x\Psi(w)=0
$$
and
$$
\pd^+_x\Psi(v+w)\ge \pd^+_x\Psi(v)-\pd^+_x\Psi(-w)=0 .
$$
Hence $\pd^+_x\Psi(v+w)=0$ and therefore $v+w$ is forward tangent to $\pd B$ at~$x$.
Similarly, $-(v+w)$ is forward tangent to $\pd B$ at~$x$, thus $v+w$ is tangent to $\pd B$ at $x$
and the lemma follows.
\end{proof}

\subsection{Steps of the proof of the main theorem}
\label{subsec:plan main}
We are now in a position to formulate two main
intermediate results in the proof of Theorem \ref{t:main}.
The first one is the following proposition,
which works in all dimensions.

\begin{prop}\label{prop:Rexists}
Let $V$ be a vector space, $\dim V=n+1\ge 3$,
and let $B\subset V$ be a convex body such that
all intersections of~$B$ with $n$-dimensional linear subspaces of $V$ are linearly equivalent.
Then for almost every $X\in\Gr_n(V)$ there exist
a vector $\nu\in V\setminus X$
and a linear map
$$
 R\co X^*\to\Hom(X,X)
$$
such that for every $\la\in X^*$
the linear operator $R_\la = R(\la)\co X\to X$ satisfies:
\begin{enumerate}
\item $\tr R_\la=0$.
\item For every $x\in \pd B\cap X$,
the vector $R_\la(x) + \la(x)\nu$ is tangent to $\pd B$ at~$x$.
\end{enumerate}
\end{prop}

Note that the assumptions of Proposition \ref{prop:Rexists}
imply that $B$ contains 0 in its interior.

The proof of Proposition \ref{prop:Rexists} is given in 
Section \ref{sec:proofRexists}.
In the case when $B$ is smooth, the argument is a simple
differential geometric calculation, see \S\ref{subsec:Rexists smooth}.
In the general case the proof simulates the same calculation
replacing derivatives by suitable partial limits.
We make use of the fact that Lipschitz functions
(in particular, convex ones) are differentiable almost everywhere,
this is the reason why ``almost every $X$'' appears in the statement.

Let $V, B, X$ and $R$ be as in Proposition~\ref{prop:Rexists} and denote $K = B \cap X$.
Proposition \ref{prop:Rexists}(2) implies the following property, which plays
a central role in subsequent arguments:
\begin{equation} \label{e:tangency assumption}
    \text{for every } \la \in X^* \text{ and } x\in \pd K \cap \ker\la, \text{ the vector } R_\la(x) \text{ is tangent to }\pd K \text{ at~$x$}.
\end{equation}
Indeed, if $\la(x)=0$ then the vector $R_\la(x) + \la(x)\nu$ from Proposition \ref{prop:Rexists}(2)
equals $R_\la(x)$ and hence belongs to~$X$,
therefore this vector is also tangent to~$\pd K = \pd B \cap X$.
Note that \eqref{e:tangency assumption} does not involve
the $(n+1)$-dimensional space $V$ and the body $B$, i.e.
it is a statement about $K$ and $R$ that live in the $n$-dimensional space~$X$.

\smallskip

The second step in the proof of Theorem \ref{t:main}
is  the following proposition.

\begin{prop}\label{prop:Rtangent}
Let $X$ be a vector space, $\dim X=3$,
and let $K\subset X$ be a convex body with $0$ in its interior.
Let $R\co X^*\to\Hom(X,X)$ be a linear map such that
for every $\la\in X^*$ the map $R_\la=R(\la)$ satisfies
$\tr R_\la=0$ and \eqref{e:tangency assumption} holds.
Then $R_\la(x)$ is tangent to $\pd K$ for all $x\in \pd K$.
\end{prop}

The following view on Proposition \ref{prop:Rtangent} is helpful.
For each $\la\in X^*$, we regard the operator $R_\la\in\Hom(X,X)$
as a linear vector field on~$X$.
Proposition \ref{prop:Rtangent} asserts that this vector field
is tangent to~$\pd K$.
Therefore the 1-parameter subgroup $\{e^{tR_\la}\}_{t\in\R}$
generated by $R_\la$ is contained in the group of
linear self-equivalences of~$K$.
Since $K$ is 3-dimensional, this leaves us with three possibilities:
either $K$ is an ellipsoid and all operators $R_\la$ 
are skew-symmetric with respect to the inner product corresponding
to~$K$,
or $K$ is an affine body of revolution and all operators $R_\la$
are proportional to the infinitesimal generator of
the respective group of affine rotations,
or $R=0$.

In the proof of Proposition \ref{prop:Rtangent} we handle
the three cases separately
but the unified formulation allows us to deduce the main theorem
without case chasing.
In some of the cases we show directly that $K$ is an ellipsoid
and finish the proof with the following lemma.

\begin{lemma}\label{l:Rtangent for ellipsoid}
Proposition \ref{prop:Rtangent} holds true if $K$ is an ellipsoid.
\end{lemma}

\begin{proof}
Let $\langle\cdot,\cdot\rangle$ be the inner product associated to~$K$,
and $|\cdot|$ the respective Euclidean norm.
We have to show that, under the assumptions of Proposition \ref{prop:Rtangent},
$\langle R_\la(v),v\rangle=0$ for all $\la\in X^*$ and $v\in X$ with $|v|=1$.
Fix $v\in X$ such that $|v|=1$ and observe that the expression
$\langle R_\la(v),v\rangle$ is linear in $\la$, therefore
it suffices to verify the identity for all $\la$ from some basis of $X^*$.
Pick an orthonormal basis $e_1,\dots,e_n$ in $X$
such that $e_n=v$, and let $f_1,\dots,f_n\in X^*$ be the dual basis.
Then $\langle R_{f_i}(e_n),e_n\rangle=0$ for all $i\le n-1$
by the assumption (2) of Proposition \ref{prop:Rtangent}.
Similarly, $\langle R_{f_n}(e_i),e_i\rangle=0$ for all $i\le n-1$.
Since $\tr R_{f_n}=0$, it follows that $\langle R_{f_n}(e_n),e_n\rangle=0$.
Thus $\langle R_{f_i}(v),v\rangle=\langle R_{f_i}(e_n),e_n\rangle=0$
for all $i=1,\dots,n$ and the lemma follows.
\end{proof}

The proof of Proposition \ref{prop:Rtangent} occupies
Sections \ref{sec:degenerate} and~\ref{sec:non-degenerate}.
We divide the problem into degenerate and non-degenerate cases
(Definition~\ref{d:degenerate point}) and handle them using different structures
arising from the tensor~$R$.
For the degenerate case (Section~\ref{sec:degenerate})
we construct a collection of quadratic vector fields tangent to $\pd K$
(Lemma \ref{l:VRtangent}) and analyze them by algebraic methods.
The degeneration implies certain cancellations that lead to the desired
conclusion (Proposition \ref{p:all degenerate cases}).

The non-degenerate case (Section \ref{sec:non-degenerate}) is the most interesting one.
In this case $R$ determines $K$ uniquely up to a homothety
and we show that $K$ is an ellipsoid.
The key construction is given in Lemma~\ref{l:flow of planes}, which shows that $R$
induces a flow on the space of 2-dimensional cross-sections of~$K$
with trajectories consisting of linearly equivalent cross-sections.
Furthermore, the non-degeneracy assumption implies that
cross-sections corresponding to fixed points of the flow are ellipses.
Then, with some help from two-dimensional topological dynamics (Lemma~\ref{l:fixedpts}),
we deduce that $K$ has a one-parameter family of ellipses among its cross-sections.
The way how $R$ determines $K$ implies that $\pd K$ is a real analytic surface, and then is is not hard to show
that $K$ is an ellipsoid and finish the proof,
see the end of Section \ref{sec:non-degenerate}.

\subsection{Proof of Theorem \ref{t:main} assuming
Propositions \ref{prop:Rexists} and \ref{prop:Rtangent}.}
\label{subsec:proof main}

The final step of the proof of Theorem \ref{t:main} is to deduce
it from the two above propositions. 

Let $B\subset V=\R^4$ be a convex body satisfying the assumptions of Theorem \ref{t:main}.
As explained after Problem \ref{problem with body},
$B$ is symmetric with respect to~0.
Let $\|\cdot\|$ be the norm with unit ball~$B$.

Let $X$, $\nu$ and $R$ be as in Proposition \ref{prop:Rexists},
and let $K=B\cap X$ be the corresponding cross-section.
As shown in \S\ref{subsec:plan main}, $K$ and $R$ satisfy \eqref{e:tangency assumption}
and therefore Proposition \ref{prop:Rtangent} applies.
Fix $x\in \pd K$ and pick $\la\in X^*$ such that $\la(x)\ne 0$.
Let $w_1=R_\la(x)$ where $R_\la=R(\la)$ and $w_2=w_1+\la(x)\nu$.
Propositions \ref{prop:Rexists} and \ref{prop:Rtangent} imply that
both $w_1$ and $w_2$ are tangent to $\pd B$ at~$x$.
Since $\nu$ is a linear combination of $w_1$ and $w_2$,
namely $\nu = \frac1{\la(x)}(w_1-w_2)$,
Lemma \ref{l:tangent linear} implies that 
$\nu$ is also tangent to $\pd B$ at~$x$.

Since $x$ is an arbitrary point of $\pd K=\pd B\cap X$,
it follows that $B$ is contained in the cylinder $K+\R\nu$.
Let $P_X\co V\to X$ be the projection along $\nu$,
that is, $P_X$ is the unique linear map from $V$ to $X$
such that $P_X|_X=\id_X$ and $P_X(\nu)=0$.
The fact that $B\subset K+\R\nu$ implies that $P_X(B)\subset B$.
Equivalently, the projector $P_X$ does not increase the norm,
that is, $\|P_X(v)\|\le \|v\|$ for all $v\in V$.

Recall that $X$ in Proposition \ref{prop:Rexists}
can be chosen arbitrarily from a set of full measure in $\Gr_3(V)$
and observe that the existence of a norm non-increasing projector
is a closed condition. 
Thus for every hyperplane $X\in\Gr_3(V)$
there exist a projector $P_X\co V\to X$ onto $X$
not increasing the norm.

To finish the proof we apply the Blaschke-Kakutani
characterization of ellipsoids \cite{Kak}.
We use the following formulation that can be found in e.g.~\cite{Gruber}:

\begin{theorem}[{\cite{Kak}, \cite{Gruber}*{Theorem 12.5}}]
\label{t:kakutani}
Let $(V,\|\cdot\|)$ be a finite-dimensional Banach space
and $2\le k < \dim V$.
Suppose that for every $k$-dimensional linear subspace
$X\in\Gr_k(V)$
there exists a linear projector $P_X$ from $V$ onto $X$
not increasing the norm. Then the norm $\|\cdot\|$ is Euclidean.
\end{theorem}

Applying Theorem \ref{t:kakutani} to $V=\R^4$,
the norm $\|\cdot\|$ associated to $B$
and $k=3$, we conclude that $B$ 
is an ellipsoid and complete the proof of Theorem \ref{t:main}.
\qed

\section{Vector fields tangent to convex surfaces}
\label{sec: polynomial vector fields}

We often interpret maps from a vector space $X$ to itself as vector fields on~$X$.
A map $W\co X\to X$ is called a
\textit{homogeneous polynomial vector field of degree $k$} if the coordinates
of $W(x)$, $x\in X$, are homogeneous polynomials
of degree $k$ in coordinates of~$x$.
Clearly this property does not depend on the choice of coordinates.
Homogeneous polynomial vector fields of degree 1, 2 and 3 are
called {\it linear}, {\it quadratic} and {\it cubic}, respectively.

A \textit{trajectory} of a vector field $W\co X\to X$ is a curve
$\ga\co(a,b) \subseteq \R \to X$
such that $\dot\ga(t) = W(\ga(t))$ for all $t\in(a,b)$.
We say that a vector field $W$ is (forward) tangent to $\pd K$,
where $K\subset X$ is a convex body, if $W(x)$ is (forward) tangent
to $\pd K$ at $x$ for every $x\in\pd K$.

In this section we prove a number of technical facts
about polynomial vector fields tangent to convex hypersurfaces.
The most essential ones are Lemma \ref{l:cubic vector field in the plane} and Lemma \ref{l:fixedpts}.

\begin{lemma}\label{l:tangent vector field}
Let $X$ be a vector space, $K\subset X$ a convex body with 0 in the interior,
and $\Psi$ the Minkowski norm associated to~$K$.
Let $W\co X\to X$ be a homogeneous polynomial vector field
forward tangent to $\pd K$.
Then $\Psi$ is constant along every trajectory of~$W$.
\end{lemma}

\begin{proof}
By Lemma \ref{l:tangent cone} we have
$\pd^+_x\Psi(W(x))=0$ for all $x\in\pd K$.
By homogeneity, this identity extends to all $x\in X$.
Let $\ga\co(a,b)\to X$ be a trajectory of $W$.
Then $\pd^+_{\gamma(t)}\Psi(\dot\ga(t))=0$ for all~$t$,
hence the function $\Psi\circ\ga$ has zero right derivative
everywhere. Since this function is locally Lipschitz,
it follows that it is constant.
\end{proof}

The next lemma is a generalization of Lemma \ref{l:tangent vector field}
used in Section \ref{sec:degenerate}.

\begin{lemma} \label{l:ae tangent vector field}
Let $X$ be a vector space, $K\subset X$ a convex body with 0 in the interior,
and $\Psi$ the Minkowski norm associated to~$K$.
Let $W\co X\to X$ be a homogeneous polynomial vector field
and $f\co X\to\R$ a nonzero homogeneous polynomial function.
Assume that the vector field 
$fW$ (that is, $x\mapsto f(x)W(x)$)
is forward tangent to~$\pd K$. 
Then $\Psi$ is constant along every trajectory of~$W$
and therefore $W$ is tangent to $\pd K$.
\end{lemma}

\begin{proof}
By Lemma~\ref{l:tangent vector field}, $\Psi$ is constant along the trajectories of~$fW$,
but trajectories of $W$ can be different because of zeroes of~$f$.
Let $\ga\co(a,b)\to X$ be a trajectory of~$W$ and consider the set
$$
 T := \{ t\in(a,b) : f(\gamma(t))=0 \}.
$$
Since $f\circ\gamma$ is a real analytic function, 
we have either $T=(a,b)$ or $T$ is a discrete subset of $(a,b)$.

First consider the case when $f\circ\ga$ is not everywhere zero on $(a,b)$.
Then $T$ is discrete: for any $t_0\in(a,b)$ there exists $\ep>0$
such that $f\circ\ga$ has no zeroes on $(t_0-\ep,t_0)\cup(t_0,t_0+\ep)$.
This implies that $\ga|_{(t_0-\ep,t_0)}$ and $\ga|_{(t_0,t_0+\ep)}$
coincide with trajectories of $fW$ up to a change of parametrization.
By Lemma \ref{l:tangent vector field} applied to $fW$
it follows that $\Psi\circ\ga$ is constant on
each of the intervals $(t_0-\ep,t_0)$ and $(t_0,t_0+\ep)$.
By continuity both constants equal $\Psi(\ga(t_0))$,
therefore $\Psi\circ\ga$ is constant on $(t_0-\ep,t_0+\ep)$.
Thus we have shown that $\Psi\circ\ga$
is locally constant and hence constant on $(a,b)$.

Now consider the case when $f\circ\ga\equiv 0$.
Fix $t_0\in(a,b)$ and let $p=\ga(t_0)$.
Since $f$ is a nonzero polynomial, the set of its zeroes
is nowhere dense in~$X$.
Hence there is a sequence $\{p_i\}\subset X$
such that $p_i\to p$ and $f(p_i)\ne 0$ for all~$i$.
Let $\ga_i\co(a_i,b_i) \to X$ be the maximal trajectory of $W$
with initial data $\ga_i(t_0)=p_i$.
Then $\ga_i$ pointwise converge to $\ga$ on $(a,b)$.
Since $f(p_i)\ne 0$, each $\ga_i$ falls under
the above case where $f\circ\ga_i$ is not everywhere zero,
thus $\Psi\circ\ga_i$ is constant for every~$i$.
Passing to the limit as $i\to\infty$ we conclude that
$\Psi\circ\ga$ is constant as well.

To prove that $W$ is tangent to $\pd K$, consider $x\in\pd K$
and let $\ga\co(-\ep,\ep)\to X$ be a trajectory of $W$
with initial data $\ga(0)=x$. Since $\Psi\circ\ga$ is constant,
we have $\ga(t)\in\pd K$ for all~$t$.
Applying Definition \ref{d:forward tangent} to
sequences $x_i=\ga(\frac 1i)$ and $c_i=i$, $i=1,2,\dots$,
yields that $W(x)=\dot\ga(0)$ is forward tangent to $\pd K$ at~$x$.
Considering $x_i=\ga(-\frac 1i)$ similarly shows that $-W(x)$
is forward tangent to $\pd K$ at $x$.
Thus $W$ is tangent to $\pd K$.
\end{proof}

\begin{lemma}\label{l:linear tangent vector field}
Let $X$ be a vector space, $K\subset X$ a convex body with 0 in the interior,
and $L\co X\to X$ a linear vector field forward tangent to $\pd K$.
Then $\tr L=0$ and $e^{tL}\in\I(K)$ for all $t\in\R$,
where $e^{tL}$ is the matrix exponential of $tL$
and $\I(K)$ is the group of linear self-equivalences of $K$,
see \eqref{e:defI(K)}.
\end{lemma}

\begin{proof}
Recall that trajectories of $L$ have the form $t\mapsto e^{tL}(p)$
where $p\in X$ is the initial value.
By Lemma \ref{l:tangent vector field},
every trajectory is either contained in~$K$ or disjoint with~$K$.
Therefore $e^{tL}(K)=K$ for all $t\in\R$, thus $e^{tL}\subset\I(K)$.
Since $e^{tL}(K)=K$, the operator $e^{tL}$ is volume-preserving,
therefore $\det e^{tL}=1$ and $\tr L=0$.
\end{proof}

\begin{lemma} \label{l:cubic vector field in the plane}
Let $K\subset\R^2$ be a convex body with $0$ in the interior
and $W$ a nonzero homogeneous polynomial vector
field on $\R^2$ of degree $\deg W\le 3$. 
Suppose that $W$ is tangent to $\pd K$.
If $\deg W = 3$, assume in addition that $W$ vanishes at some point
other than the origin.
Then $K$ is an ellipse.  
\end{lemma}

\begin{proof}
Clearly $\deg W\ne 0$.
Consider the three cases according to the possible values of~$\deg W$.

{\bf Case 1:} $\deg W = 1$.
By Lemma \ref{l:linear tangent vector field}, the self-equivalence
group $\I(K)$ contains a one-parameter subgroup $G = \{e^{tW}\}_{t\in\R}$.
As $\I(K)$ is a compact subgroup of $\GL(2)$, this is possible only if $G$ is conjugate to $SO(2)$
and thus $K$ is an ellipse.

{\bf Case 2:} $\deg W = 2$.
Consider a function $f\co\R^2\to\R$ defined by
$f(x)=x\wedge W(x)$ where the $\wedge$-product
is just the determinant of the $2\times 2$ matrix
composed of vectors $x$ and $W(x)$.
Observe that $f$ is a homogeneous polynomial of degree~3,
in particular, $W(-x)=-W(x)$ for all $x\in\R^2$.
Hence $f$ attains values of opposite signs on $\pd K$,
therefore $f(p)=0$ for some $p\in\pd K$.

The vector $W(p)$ is collinear to $p$
and tangent to $\pd K$ at the same time,
hence $W(p)=0$.
By homogeneity, $W$ vanishes on the entire line $\ell=\{tp\mid t\in\R\}$.
Therefore $W$ can be decomposed as a product $W=\la L$
where $\la\co\R^2\to\R$ is a linear function with $\ker\la=\ell$
and $L\co\R^2\to\R^2$ is a linear vector field.
By Lemma \ref{l:ae tangent vector field}, $L$ is tangent to $\pd K$
and then the result of Case~1 implies that $K$ is an ellipse.

{\bf Case 3:} $\deg W = 3$.
By assumption there exists $p\in\R^2\setminus\{0\}$ such that $W(p)=0$.
As in Case~2, $W$ vanishes at the line $\ell=\{tp\mid t\in\R\}$
and hence $W=\la Q$ where $\la$ is a linear function
and $Q$ is a quadratic vector field.
By Lemma \ref{l:ae tangent vector field}, $Q$ is tangent to $\pd K$
and then the result of Case~2 implies that $K$ is an ellipse.
\end{proof}

In Section \ref{sec:non-degenerate} we will need some facts about quadratic vector fields
on smooth convex surfaces. The following dynamical systems terminology will be handy.

\begin{defn}
Let $S$ be a smooth surface (i.e., a two-dimensional smooth manifold)
and $W$ a complete smooth vector field on~$S$.
An \textit{orbit} of $W$ is a subset of $S$ of the form $\{\ga(t):t\in\R\}$
where $\ga$ is a trajectory of~$W$.
Sets of the form $\{\ga(t):t\ge 0\}$ and $\{\ga(t):t\le 0\}$
are called \textit{half-orbits}
(forward and backward ones, respectively).
A \textit{closed orbit} is an orbit corresponding to a non-constant
periodic trajectory.
By abuse of notation, a \textit{fixed point} of $W$ is the same as a zero of $W$,
i.e., a point $p\in S$ such that $W(p)=0$.
\end{defn}

The following routine lemma is a preparation to Lemma~\ref{l:fixedpts}.

\begin{lemma}\label{l:closed orbit}
Let $W$ be a smooth vector field on a closed smooth surface $S$
and let $p\in S$ be an isolated zero of~$W$.
Assume that $p$ does not belong to the closure of any non-constant
orbit of~$W$.

Then every neighborhood $U$ of $p$ contains a closed orbit
of~$W$ separating $p$ from $S\setminus U$.

Furthermore $\operatorname{ind}_p(W)=1$ where $\operatorname{ind}_p(W)$
is the index of $W$ at~$p$ (as in the Poincar\'e-Hopf Index Theorem,
see e.g.\ \cite{GP}*{Chapter~3, \S5}).
\end{lemma}

\begin{proof}
We may assume that $U$ is a topological disc 
and that it does not contain zeroes of~$W$ other than $p$.
Let $U_1\ni p$ be a smaller neighborhood whose closure $\overline U_1$
is contained in~$U$.

First we show that there exist a non-constant half-orbit of $W$
contained in $\overline U_1$.
Pick a sequence $\{p_i\}_{i=1}^\infty\subset U_1\setminus\{p\}$
such that $p_i\to p$. 
For each $i$, let $\ga_i$ be the trajectory of $W$
with $\ga_i(0)=p_i$ and let $(-a_i,b_i)$ be the maximal interval of $\R$
containing 0 and contained in the set $\{t\in\R:\ga(t)\in U_1\}$.
Since $p_i\to p$ and $p$ is a fixed point of $W$, we have $a_i\to+\infty$.
If $a_i=+\infty$ for some $i$, then $\{\ga_i(t):t\le 0\}$
is a desired half-orbit.
Otherwise we have a sequence of points $\{\ga_i(-a_i)\}$
in the boundary of $U_1$.
Let $q$ be a partial limit of this sequence
and let $\ga$ be the trajectory of $W$ with $\ga(0)=q$.
Then $\ga$ is the limit of a subsequence of trajectories $t\mapsto \ga_i(t-a_i)$.
Since $a_i\to +\infty$, it follows that
$\ga(t)\in\overline U_1$ for all $t\ge 0$,
so $\{\ga(t):t\ge 0\}$ is a desired half-orbit.

Thus we have proved the existence of a non-constant half-orbit of $W$
in $\overline U_1$.
The Poincar\'e-Bendixson Theorem (see e.g. \cite{KH}*{Theorem 14.1.1})
implies that the closure of this half-orbit contains a fixed point or a closed orbit.
The case of a fixed point is ruled out by the assumption of the lemma
since $p$ is the only fixed point in~$U$.
Hence $U$ contains a closed orbit of~$W$.

Let $D\subset U$ be the disc bounded by this closed orbit.
Then $D$ is invariant under the flow generated by $W$.
Hence, by Brouwer's fixed-point theorem, $D$ contains a fixed point of the flow.
Since $p$ is the unique fixed point in~$U$, it follows that $p\in D$,
hence the closed orbit separates $p$ from $S\setminus U$.

To prove the second claim of the lemma, calculate the index $\operatorname{ind}_p(W)$ in a local 
coordinate chart which maps $p$ to $0\in\R^2$ and $D$ to the standard disc in $\R^2$.
In this chart the relation $\operatorname{ind}_p(W)=1$ is immediate
from the definition of the index, see \cite{GP}*{p.~134}.
\end{proof}

\begin{lemma}\label{l:fixedpts}
Let $W$ be a nonzero
quadratic homogeneous vector field on a vector space $Y\simeq\R^3$,
and let $S\subset Y$ be a smooth closed strictly convex surface
symmetric with respect to~0.
(The strict convexity means that $S$ is a boundary
of a convex body and $S$ contains no straight line segments).

Assume that $W$ is tangent to~$S$.
Then there exist a non-constant orbit of $W$ in $S$
whose closure contains a zero of $W$.
\end{lemma}

\begin{proof}
Suppose to the contrary that
no zero of $W$ belongs to the closure of a non-constant orbit.
We restrict $W$ to $S$ and regard it as a vector field on~$S$.
Note that $S$ is diffeomorphic to the 2-sphere.

{\bf First consider the case when $W$ has finitely many zeroes.}
By the Poincar\'e-Hopf Theorem (see e.g.\ \cite{GP}*{Chapter~3, \S5}),
the sum of indices $\operatorname{ind}_p(W)$
where $p$ ranges over all zeroes of $W$, equals~2.
By Lemma \ref{l:closed orbit} each index equals~1,
hence $W$ has exactly two zeroes on~$S$. 
Denote one of them by~$p$, then the other one is $-p$
due to the symmetry of $S$ and~$W$.

Choose a basis $e_1,e_2,e_3$ in $Y$ such that $e_3=p$ and
$e_1,e_2$ are parallel to the tangent plane to $S$ at~$p$.
The choice of $e_1$ and $e_2$ will be refined later.
Identify $Y$ with $\R^3$ by means of this basis
and denote the respective coordinates by $x,y,z$.
We call the $xy$-plane \textit{horizontal}.

We write $W$ as a function of coordinates $x,y,z$:
$$
W(x,y,z) = (W_1(x,y,z), W_2(x,y,z), W_3(x,y,z))
$$
where $W_i(x,y,z)$ are quadratic forms,
and decompose $W_i$ as
$$
 W_i(x,y,z) = Q_i(x,y) + z\cdot L_i(x,y)
$$
where $Q_i$ and $L_i$ are a quadratic form and a linear function
on the $xy$-plane, $i=1,2,3$.
The lack of $z^2$ in the formula follows from the fact that
$W(0,0,1)=W(p)=0$.

Since the tangent plane to $S$ at $p=(0,0,1)$ is horizontal,
$S$ is smooth and $W$ is homogeneous, we have
$$
 \frac{W_3(x,y,1)}{\|W(x,y,1)\|} \to 0 \quad\text{as $x,y\to 0$}.
$$
This implies that $L_3=0$, so $W_3(x,y,z)=Q_3(x,y)$ for all $x,y,z$.

Observe that $Q_3$ cannot be positive or negative definite,
as otherwise the $z$-coordinate would be strictly monotone along
a periodic trajectory provided by Lemma~\ref{l:closed orbit}.
Therefore $Q_3$ vanishes at some nonzero horizontal vector,
and we now require that the basis is chosen so that
$e_2$ is such a vector.
Then $W_3(0,y,z)=Q_3(0,y)=0$ for all $y,z$.

Consider the restriction of the quadratic form $W_1$ to the $yz$-plane.
We are going to show that this restriction is semi-definite,
vanishing only on the line $\{x=y=0\}$ containing~$p$.
Suppose to the contrary that $W_1$ vanishes on some other line in the
$yz$-plane. Then there is a point $q=(0,y_0,z_0)\in S$
such that $q\notin\{p,-p\}$ and $W_1(q)=0$.
Observe that $|z_0|<1$ since all points of $S$ except $p$ and $-p$ lie in the open strip $\{-1<z<1\}$
due to the strict convexity of $S$.

Since $W_1(q)=0$ and $W_3$ vanishes on the $yz$-plane, $W(q)$ is proportional to $(0,1,0)$.
Note that $W(q)\ne 0$ since $p$ and $-p$ are the only zeroes of~$W$.
Recall that $W(q)$ is tangent to $S$ at~$q$, hence the straight line
$$
 \{ q + tW(q) : t\in\R \} = \{ (0,y_0+t,z_0) : t\in\R\}
$$
does not intersect the interior of the body bounded by~$S$.
However this line contains the point $(0,0,z_0)$ which belongs
to the interior of the body (since $|z_0|<1$).
This contradiction shows that the restriction of $W_1$ to the $yz$-plane
is semi-definite. Hence it has a constant sign,
either positive or negative, on the set $\{x=0,\,y\ne 0\}$.
Replacing, if necessary, $W$ with $-W$, we can assume that the sign is positive.

Consider the disc (``half-sphere'') $D = S \cap \{x\ge 0 \}$.
The positivity of $W_1$ on $\{x=0,\,y\ne 0\}$
implies that  $W$ points inwards $D$ everywhere on the boundary of~$D$
except the points $p$ and $-p$ where $W$ vanishes.
This implies that any forward half-orbit of~$W$ starting at a point of~$D$
never leaves~$D$.
However by Lemma \ref{l:closed orbit} there exists a periodic trajectory of $W$
separating $p$ from $-p$. This trajectory enters and leaves $D$ infinitely many times,
a contradiction.
This proves Lemma~\ref{l:fixedpts}
in the case when $W$ has finitely many zeroes.

{\bf Now consider the case when $W$ has infinitely many zeroes on~$S$.}
Identify $Y$ with $\R^3$ and denote the set of zeroes of $W$ on $S$
by~$\Gamma$.
Then $\Gamma$ is the intersection of $S$ with zero sets of the
three coordinate functions of~$W$.
These functions are quadratic forms on $\R^3$,
so each of the three zero sets is either all of $\R^3$
or a conical surface representing a quadric in the projective plane.
The intersection of plane quadrics is infinite only if it is
a straight line or all the quadrics coincide.
Hence $\Gamma$ is the intersection of $S$ with either
a plane, or a union of two planes, or an elliptic cone.
In any case $S\setminus\Gamma$ has a component $U$
homeomorphic to the open disc.
Pick a non-constant orbit of $W$ contained in $U$
and apply the Poincar\'e-Bendixson Theorem as in
the proof of Lemma \ref{l:closed orbit}.
Since by our assumption the closure of the orbit
cannot contain a fixed point, it contains a closed orbit.
This closed orbit bounds a disc $D\subset U$ which is
invariant under the flow generated by $W$ and hence
contains a fixed point of the flow by Brouwer's fixed-point theorem.
This fixed point is a zero of $W$ that does not belong to~$\Gamma$,
a contradiction. This finishes the proof of Lemma~\ref{l:fixedpts}.
\end{proof}

\section{Proof of Proposition \ref{prop:Rexists}}
\label{sec:proofRexists}

In this section we prove Proposition~\ref{prop:Rexists}. To facilitate understanding,
we begin in \S\ref{subsec:Rexists smooth}
with a short argument proving the proposition in the smooth case, where it holds for all rather than for almost all $X\in\Gr_n(V)$.
Then in \S\ref{subsec:nu} we explain the choice of $X$ and $\nu$,
and finally in \S\ref{subsec:R_lambda} we construct the desired tensor $R$.
We fix $V$ and $B$ satisfying the assumptions of Proposition \ref{prop:Rexists}
for the rest of the section.
Recall that $B$ contains 0 in its interior.

\subsection{Smooth case}\label{subsec:Rexists smooth}
First we assume that  $\pd B$ is $C^\infty$-smooth
and prove Proposition~\ref{prop:Rexists} under this assumption.

Fix $X\in\Gr_n(V)$ and $\nu\in V\setminus X$.
We decompose $V$ as $X\oplus\R\nu\simeq X\times\R$ and parametrize
a neighborhood of $X$ in $\Gr_n(V)$ by elements of $X^*$ as follows:
to each $\la\in X^*$ we associate its graph $H_\la\subset V$.
That is,
\begin{equation}\label{e:Hla}
 H_\la = \{ x + \la(x)\nu \mid x\in X \} .
\end{equation}
Note that $H_\la\in\Gr_n(V)$ for all $\la$ and $H_0=X$.
Denote $X\cap B$ by $K$.

Since all cross-sections of $B$ are linearly equivalent,
for every $\la\in X^*$ there exists a linear map $F_\la\co X\to V$
such that $F_\la(X) = H_\la$ and $F_\la(K) = H_\la \cap B$.
The smoothness of $\Psi_B$ implies (see \cite{BM}*{\S2.1})
that one can choose the family $\{F_\la\}$ to be smooth
in a neighborhood of $\la=0$
and such that $F_0=\id_X$.
Since $F_\la(X)=H_\la$, $F_\la$ can be written in the form
\begin{equation}\label{e:FlaGla}
 F_\la(x) = G_\la(x) + \la(G_\la(x))\cdot \nu
\end{equation}
for $G_\la\in\Hom(X,X)$ defined by $G_\la=\pr_\nu\circ F_\la$
where $\pr_\nu\co V\to X$ is the projection along $\nu$.

Now define $R$ as the differential of the map $\la\mapsto G_\la$ at $\la=0$.
By construction, $R$ is a linear map from $X^*$ to $\Hom(X,X)$.
We denote $R(\la)$ by $R_\la$ for all $\la\in X^*$.

Since $F_\la(K)$ is a cross-section of $B$, we have $F_\la(x)\in \pd B$
for all $x\in \pd K$.
Therefore the vector $\frac d{dt}\big|_{t=0} F_{t\la}(x)$ is tangent to $B$
at $F_0(x)=x$. From \eqref{e:FlaGla} we have
$$
  \frac d{dt}\Big|_{t=0} F_{t\la}(x)
  = \frac d{dt}\Big|_{t=0} G_{t\la}(x) + \la(G_0(x))\nu
  = R_\la(x) + \la(x) \nu .
$$
Thus for every $x\in \pd K$ the vector $R_\la(x) + \la(x) \nu$
is tangent to $B$ at~$x$.
This is the second property of $R_\la$ claimed in Proposition \ref{prop:Rexists}.

To obtain the remaining property $\tr R_\la=0$,
we have to adjust the choice of~$\nu$.
In the next subsection we describe a construction for $\nu$
based on the derivative of the cross-section area
and prove the identity $\tr R_\la=0$ in Lemma \ref{l:trace0}.
Alternatively, in the smooth case one can use the following short argument.

Having constructed $R$ as above, let us replace $\nu$ by
$\nu'=\nu+w$, where $w\in X$, and adjust $R$ accordingly.
Namely for $\la\in X^*$ define $R'_\la\in\Hom(X,X)$ by
$$
 R'_\la(x) = R_\la(x) - \la(x)w .
$$
Then the linear map $R'\co X^*\to \Hom(X,X)$
is defined by $R'(\la)=R'_\la$ for all $\la\in X^*$.
By construction,
$$
 R'_\la(x) + \la(x)\nu' = R_\la(x) + \la(x)\nu \quad \text{ for all } x \in X,
$$ 
hence (2)~in~Proposition \ref{prop:Rexists}
is satisfied for $R'$ and~$\nu'$.
For the traces we have
$$
 \tr R'_\la = \tr R_\la - \la(w) .
$$
Since the map $\la\mapsto\tr R_\la$ is linear, there exists $w\in X$
such that $\la(w) = \tr R_\la$ for all $\la\in X^*$.
Choosing this $w$ for the above construction and 
substituting $\nu'$ and $R'$
for $\nu$ and $R$ finishes the proof of Proposition \ref{prop:Rexists}
in the smooth case.

\subsection{Choice of $X$ and $\nu$}\label{subsec:nu}
Now we return to the general case where no smoothness is assumed.
Pick an auxiliary Euclidean metric on $V$
and consider the cross-section area function
\begin{equation}\label{e:cross area}
X\mapsto \vol_n(B\cap X), \qquad X\in\Gr_n(V),
\end{equation}
where $\vol_n$ is the $n$-dimensional Euclidean  volume.
This function is Lipschitz and therefore,
by Rademacher's Theorem, it is differentiable almost everywhere.
Our plan is to prove the claims
of Proposition \ref{prop:Rexists}
for any $X\in\Gr_n(V)$ where this function
is differentiabe.
Though $\vol_n$ depends on the choice of the auxiliary metric,
the differentiability property does not.

In what follows we use an affine-invariant version of the cross-section area.
It is a function on the exterior product $\Lambda^n V$
defined in \eqref{e:affine area} below.

Since $\dim V=n+1$, every $n$-vector $\sigma\in \Lambda^n V$
can be written in the form
$
 \sigma = v_1\wedge\dots\wedge v_n
$
for some $v_1,\dots,v_n\in V$.
If $\sigma\ne 0$ then $v_1,\dots,v_n$ are linearly independent.
In this case we denote by  $\Pi_\sigma$ the linear span
of $v_1,\dots,v_n$,
and by $\vol_\sigma$ the Haar measure on
the hyperplane $\Pi_\sigma$ normalized so
that the measure of the parallelotope with edges $v_1,\dots,v_n$
equals~1.
One easily sees that $\Pi_\sigma$ and $\vol_\sigma$
do not depend on the choice of vectors $v_1,\dots,v_n$
representing~$\sigma$.
Now define
\begin{equation}\label{e:affine area}
 A(\sigma) = \vol_\sigma(B\cap \Pi_\sigma)^{-1}.
\end{equation}
The resulting function
$A \co \Lambda^n V \setminus\{0\} \to \R_+$
is positively $1$-homogeneous and it can be extended to~0
by setting $A(0)=0$.
If $V$ is equipped with a Euclidean metric then
$A(\sigma) = \|\sigma\|\vol_n(B\cap \Pi_\sigma)^{-1}$,
where $\|\sigma\|$ is the Euclidean norm of $\sigma$.
Hence $A$ is differentiable at $\sigma$
if and only if the cross-section area \eqref{e:cross area}
is differentiable at~$\Pi_\sigma$.
One can see that $A$ is essentially the Busemann-Hausdorff area
(\cite{Bus47}, see also \cite{Thompson}*{Chapter 7})
of the Minkowski norm associated to~$B$,
except for a normalization constant, which we omit for convenience.

Let $\sigma$ be a differentiability point of $A$
and $X=\Pi_\sigma$.
We now construct a vector $\nu$ required
in Proposition \ref{prop:Rexists}.
Consider the differential $d_\sigma A$ of $A$ at~$\sigma$.
It is a linear function on $\Lambda^n V$,
and we regard it as an exterior form on $V$,
that is, $d_\sigma A\in \Lambda^n V^*$.
Since $\dim V=n+1$, there exists a vector $\nu\in V\setminus\{0\}$
such that
\begin{equation}\label{e:nu in kernel}
d_\sigma A (\nu \wedge \beta) = 0 \qquad
\textstyle\text{for all $\beta \in \Lambda^{n-1} V$}
\end{equation}
Indeed, the exterior $n$-form $d_\sigma A$
can be written as $d_\sigma A = \ell_1\wedge\dots\wedge\ell_n$
for some $\ell_1,\dots,\ell_n\in V^*$.
In order to satisfy \eqref{e:nu in kernel}, it suffices
to pick $\nu$ from $\bigcap_{i=1}^n \ker\ell_i \setminus\{0\}$.
Such a vector exists since $\bigcap_{i=1}^n \ker\ell_i$
is at least one-dimensional.
Note that
\begin{equation}\label{e:dAsigma}
d_\sigma A(\sigma) = A(\sigma) \ne 0
\end{equation}
due to the homogeneity of $A$.
This and \eqref{e:nu in kernel} imply that $\nu\notin X$.

We fix $X\in\Gr_n(V)$ and $\nu\in V\setminus X$ 
constructed above for the rest of this section,
and denote $B\cap X$ by~$K$.

\begin{remark}
The vector $\nu$ can be characterized in Euclidean terms as follows:
If a Euclidean structure on $V$ is defined in such a way
that $\nu\perp X$, then the cross-section area \eqref{e:cross area}
has zero derivative at~$X$.
We do not prove this fact as we do not use it in this paper.
\end{remark}

The choice of $\nu$ is essential for (1)~in~Proposition \ref{prop:Rexists}.
Namely, we have the following lemma.

\begin{lemma}\label{l:trace0}
Let $X$, $K$ and $\nu$ be as above.
Let $\la\in X^*$ and $L\in\Hom(X,X)$ be such that for every $x\in \pd K$
the vector $L(x) + \la(x)\nu$ is forward tangent to $B$ at~$x$.
Then $\tr L=0$.
\end{lemma}

\begin{proof}
Pick a basis $e_1,\dots,e_n$ of $X$ and let $\sigma=e_1\wedge\dots\wedge e_n$.
We equip $X$ with a Euclidean metric such that $(e_i)$ is
an orthonormal basis.
For $t\in\R$, define $F_t\in\Hom(X,V)$ by
$$
 F_t(x) = x + t  \big(L(x) + \la(x)\nu\big)
$$
and $\sigma_t\in\Lambda^n V$ by
\begin{equation}\label{e:sigma_t}
 \sigma_t = (F_t)_* (\sigma)
 = \bigwedge_{i=1}^n \big(e_i + tL(e_i) + t\la(e_i)\nu \big) .
\end{equation}
Define $a(t)=A(\sigma_t)^{-1}$ where $A$ is the
function from \eqref{e:affine area}.
The definitions imply that
$
 a(t) = | F_t^{-1}(B) |
$ 
where $|\cdot|$ denotes the Euclidean volume in~$X\simeq\R^n$.
We are going to show that the right derivative of $a(t)$ at $t=0$ equals~0.

Let $\Psi$ be the norm on $V$ associated to~$B$.
Then $F_t^{-1}(B)$ is the unit ball of the norm
$\Psi_t:=\Psi\circ F_t$ on $X$.
Observe that for every $x\in\pd K$,
\begin{equation}\label{e:dPsiFt}
 \frac d{dt}\Big|_{t=0+} \Psi_t(x)
 = \pd^+_x\Psi\big( \tfrac d{dt}\big|_{t=0} F_t(x)\big)
 = \pd^+_x\Psi \big(L(x) + \la(x)\nu\big) = 0,
\end{equation}
where the last equality follows from Lemma \ref{l:tangent cone}
and the forward tangency assumption in the statement of the lemma.
Furthermore, \eqref{e:dPsiFt} holds for all $x\in X$ due to homogeneity.
We now express the volume of the unit ball of $\Psi_t$
in terms of its radial function:
$$
a(t) = \frac 1n \int_{S(X)} \Psi_t(x)^{-n}\, d\vol_{n-1}(x)
$$
where $S(X)$ is the unit sphere in~$X$.
Since $\Psi_t(x)$ is a convex function of $(x,t)$,
the formula has enough regularity to allow differentiation under the integral.
This and \eqref{e:dPsiFt} imply that $\frac d{dt}\big|_{t=0+} a(t) = 0$, as claimed.

On the other hand, since $a(t)=A(\sigma_t)^{-1}$ and $A$ is differentiable
at $\sigma_0=\sigma$, we have
$$
0=\tfrac d{dt}\big|_{t=0+} a(t) 
= d_\sigma (A^{-1})\big(\tfrac d{dt}\big|_{t=0}\sigma_t\big)
= -A(\sigma)^{-2}\cdot d_\sigma A\big(\tfrac d{dt}\big|_{t=0}\sigma_t\big).
$$
Thus $d_\sigma A\big(\tfrac d{dt}\big|_{t=0}\sigma_t\big)=0$.
Define
$\widetilde\sigma_t = \Lambda_{i=1}^n \big(e_i + tL(e_i) \big)$
and observe that 
$$
d_\sigma A\big(\tfrac d{dt}\big|_{t=0}\widetilde\sigma_t\big) = d_\sigma A\big(\tfrac d{dt}\big|_{t=0}\sigma_t\big) = 0
$$
because the terms involving $\nu$ in the definition \eqref{e:sigma_t} of $\sigma_t$ vanish under $d_\sigma A$
due to \eqref{e:nu in kernel}. 
Expanding the derivative of a wedge-product yields
$$
 \tfrac d{dt}\big|_{t=0}\widetilde\sigma_t = \tr L\cdot\sigma
$$
Since $d_\sigma A(\sigma)\ne 0$ by \eqref{e:dAsigma},
the last two equations imply that $\tr L=0$.
\end{proof}

\subsection{Construction of $R$}\label{subsec:R_lambda}
Let $X$ and $\nu$ be as above and $K=B\cap X$.
Recall that $\I(K)$ denotes the group of linear self-equivalences of $K \subset X$ (see \eqref{e:defI(K)}).
Since $\I(K)$ is a compact subgroup of $GL(X)$,
there is a Euclidean inner product $\langle\cdot,\cdot\rangle$
that is invariant under $\I(K)$
(see \cite{BHJM21}*{Lemma 2.2} and \cite{Gro67}*{Lemma 1}
for some explicit constructions).
We extend this inner product to $V$ in such a way that $\nu\perp X$
and $\langle\nu,\nu\rangle=1$.
Throughout the rest of this section $V$ is regarded as a Euclidean
space with this inner product.

The Euclidean structures on $X$ and $V$ induce a Euclidean inner product on
$\Hom(X,V)$ in a standard way: for $F,G\in\Hom(X,V)$,
\begin{equation}\label{e:hom product}
 \langle F, G \rangle = \tr (F\circ G^*) = \tr(G\circ F^*)
\end{equation}
where $F^*,G^*\in \Hom(V,X)$ are the adjoint operators to $F$ and $G$.
We use the Euclidean norm on $\Hom(X,V)$ 
defined by $\|F\|^2=\langle F,F\rangle$.

There is a natural right action (by composition) of $\I(K)$ on $\Hom(X,V)$.
Since all elements of $\I(K)$ are orthogonal operators on $X$, this action
is isometric: $\langle F\circ L, G\circ L\rangle = \langle F,G\rangle$
for all $F,G\in\Hom(X,V)$ and $L\in\I(K)$.

As in \S\ref{subsec:Rexists smooth},
we denote by $\pr_\nu$ the projection from $V$ to $X$ along $\nu$,
and by $H_\la$ the hyperplane corresponding to $\la\in X^*$ as in~\eqref{e:Hla}.
Since $\nu$ is now a unit normal vector to $X$, we can
rewrite the definition of $H_\la$ as follows:
\begin{equation} \label{e:Hla2}
 H_\la = \{ x\in V : \langle x, \nu \rangle = \lambda(\pr_\nu(x)) \} .
\end{equation}
Note that $H_0=X$ and the map $\lambda \mapsto H_\lambda$ 
is a homeomorphism between $X^*$ and the set of hyperplanes from
$\Gr_n(V)$ that do not contain~$\nu$.

Recall that every cross-section $B\cap H_\la$ is linearly equivalent to $K$.
We denote by $\I_\la$ the set of all linear equivalences between them,
more precisely,
$$
 \I_\la = \{ F\in\Hom(X,V) : F(K) = B\cap H_\la \} \text{ for all } \la\in X^*.
$$
Clearly $\I_\la$ is an orbit of the  aforementioned action of $\I(K)$ and the stabilizer of each $F \in \I_\la$ is trivial. 
In particular, $\I_\la$ is a compact smooth submanifold of $\Hom(X,V)$.
The set $\I_0$ is the image of $\I(K)$ under the natural embedding
$\Hom(X,X)\hookrightarrow\Hom(X,V)$.
Denote by $i_0$ the inclusion $X\hookrightarrow V$,
thus $i_0$ is the element of $\I_0$ corresponding to
the identity of~$\I(K)$.

Let us fix $\lambda \in H_0^* \setminus \{0\}$.
We are going to construct an operator $R_\lambda \in \Hom(X,X)$
satisfying the properties from Proposition \ref{prop:Rexists}.

Let $\{t_k\}_{k=1}^\infty$ be a sequence of positive reals decreasing to~$0$.
For each $k$, let $F_k\in\I_{t_k\lambda}$ be a nearest point
to $i_0$ in $\I_{t_k\lambda}$:
$$
 \|F_k - i_0\| = \inf_{F\in\I_{t_k\lambda} } \|F-i_0\|
$$
where $\|\cdot\|$ is the Euclidean norm on $\Hom(X,V)$ defined above.
Note that $F_k\ne i_0$ since $\I_{t_k\lambda}$ and $\I_0$ are disjoint.

\begin{lemma}\label{l:Fm trivia}
The sequence $\{F_k\}\subset\Hom(X,V)$ constructed above satisfies:
\begin{enumerate}
\item For every $k$ one has
$F_k-i_0\in (T_{i_0}\I_0)^\perp$
where $T_{i_0}\I_0$ is the
tangent space to $\I_0$ at $i_0$ regarded as a linear subspace of $\Hom(X,V)$,
and the orthogonal complement is taken in $\Hom(X,V)$
with respect to the inner product \eqref{e:hom product}.
\item $F_k\to i_0$ as $k\to\infty$.
\end{enumerate}
\end{lemma}

\begin{proof} Since the group $\I(K)$ acts on $\Hom(X,V)$ by isometries
and both $\I_{t_k\la}$ and $\I_0$ are orbits of this action,
$\|F_k-i_0\|$ is in fact the minimum distance between these orbits:
\begin{equation} \label{e:orbit dist}
 \|F_k-i_0\| = \dist(\I_{t_k\la},\I_0) 
 := \inf \{ \|F-G\| : F\in \I_{t_k\la}, \, G\in\I_0 \} .
\end{equation}
In particular, $i_0$ is nearest to $F_k$ among points of~$\I_0$.
Hence the derivative at $i_0$ of the function $G\mapsto \|F_k-G\|^2$
is zero along any vector from $T_{i_0}\I_0$.
This derivative  along a vector $W\in\Hom(X,V)$ 
equals $2\langle F_k-i_0,W\rangle$.
Thus $F_k-i_0 \perp W$ for all $W\in T_{i_0}\I_0$
and the first claim of the lemma follows.

As $\{F_k\}$ is a bounded sequence, it either converges to $i_0$
or has a partial limit different from~$i_0$.
Let $F_\infty\in\Hom(X,V)$ be a partial limit of $\{F_k\}$.
Passing to a subsequence we may assume that $F_k\to F_\infty$ as $k\to\infty$.
Since $F_k(K)=B\cap H_{t_k\la}$ for all $k$ and the hyperplanes $H_{t_k\la}$
converge to $H_0=X$, we have $F_\infty(K)=K$, therefore $F_\infty\in\I_0$.
By \eqref{e:orbit dist},
$$
 \|F_k-i_0\| = \dist(\I_{t_k\la},\I_0) \le \|F_k-F_\infty\|
 \xrightarrow[m \to \infty]{} 0 ,
$$
therefore $F_k\to i_0$ and hence $F_\infty=i_0$.
The second claim of the lemma follows.
\end{proof}

Passing to a subsequence we may assume that there exists a limit
\begin{equation}\label{e:defWla}
W_\lambda = \lim_{k \to \infty} \frac{F_k - i_0}{\|F_k - i_0\|} \in \Hom(X,V).
\end{equation}

\begin{lemma}\label{l:Wla trivia}
The operator $W_\la$ defined by \eqref{e:defWla} satisfies:
\begin{enumerate}
\item $W_\la\ne 0$.
\item $W_\la \in (T_{i_0}\I_0)^\perp$.
\item For every $x\in\pd K$, the vector $W_\la(x)$ is forward tangent to $\pd B$.
\item There exists $C>0$ such that
$
 \langle W_\la(x), \nu\rangle = C \la(x)
$
for all $x\in X$.
\end{enumerate}
\end{lemma}

\begin{proof}
The first claim follows from the fact that $\|W_\la\|=1$.
The second one follows from Lemma \ref{l:Fm trivia}(1).
The third one follows from the definition of forward tangency
(see Definition \ref{d:forward tangent}) applied to the sequence
$\{F_k(x)\}_{k=1}^\infty\subset\pd B$;
this sequence converges to $x$ due to Lemma \ref{l:Fm trivia}(2).

To prove the fourth claim, recall that $F_k(X)=H_{t_k\la}$.
Hence, by \eqref{e:Hla2},
$$
 \langle F_k(x),\nu \rangle = t_k\lambda(\pr_\nu(F_k(x)))
$$
for all $x\in X$. Since $\langle x,\nu \rangle=0$ and $\pr_\nu(x)=x$,
this identity can be rewritten as follows:
$$
 \langle F_k(x) - x,\nu \rangle =t_k\lambda(\pr_\nu(F_k(x)-x)) + t_k\la(x) .
$$
As $\frac{F_k(x)-x}{\|F_k-i_0\|} \to W_\la(x)$ and $t_k\to 0$,
dividing the last equality by $\|F_k-i_0\|$ and passing to the limit yields that
$$
 \langle W_\la(x),\nu \rangle = \la(x) \lim_{k\to\infty}\frac{t_k}{\|F_k - i_0\|} ,
$$
in particular the limit in the right-hand side exists.
Denote this limit by $C$, and we obtain the desired identity 
$ \langle W_\la(x), \nu\rangle = C \la(x) $.
Clearly $C\ge 0$.

It remains to prove that $C\ne 0$.
Suppose to the contrary that $\langle W_\la(x),\nu\rangle=0$ for all $x\in X$.
Hence $W_\la(X)\subset X$, so we may now regard $W_\la$
as an element of $\Hom(X,X)$ and as a linear vector field on~$X$.
By Lemma \ref{l:linear tangent vector field},
the 1-parameter subgroup $\{e^{tW_\la}\}_{t\in\R}$ of $GL(X)$ is contained in~$\I(K)$.
Therefore $W_\la$ belongs to the tangent space $T_e\I(K)$ of $\I(K)$
regarded as a linear subspace of $\Hom(X,X)$.
However by (1) and (2), $W_\la$ is nonzero and orthogonal
to this tangent space, a contradiction.
Thus $C\ne 0$ and the last claim of Lemma \ref{l:Wla trivia} follows.
\end{proof}

Now define $R_\la\in\Hom(X,X)$ by
\begin{equation}\label{e:defRla}
 R_\la = C^{-1} \pr_\nu \circ W_\la
\end{equation}
where $W_\la$ is defined by \eqref{e:defWla}
and $C$ is the constant from Lemma \ref{l:Wla trivia}(4).
By Lemma \ref{l:Wla trivia}(4) we have
$$
 R_\la(x) + \la(x)\nu = C^{-1} W_\la(x)
$$
for all $x\in X$, and by Lemma \ref{l:Wla trivia}(3)
this vector is forward tangent to $\pd B$ whenever $x\in\pd K$.
This tangency and Lemma \ref{l:trace0} imply that $\tr R_\la=0$.
In addition, Lemma \ref{l:Wla trivia}(2) implies that 
$R_\la\perp T_e\I(K)$ in $\Hom(X,X)$.

To finish the proof of Proposition~\ref{prop:Rexists}, it remains to show that $R_\la$ depends linearly on~$\la$.
This is verified in the following proposition.

\begin{prop}\label{prop:Rdef}
For every $\la\in X^*$ there exists a unique operator $R_\la \in \Hom(X,X)$
such that
\begin{enumerate}
    \item $\tr R_\lambda = 0$.
    \item  $R_\la(x) + \la(x)\nu$ is forward tangent to $\pd B$
        for all $x \in \pd K$.
    \item $R_\la \perp T_e\I(K)$ in $\Hom(X,X)$.
\end{enumerate}
Moreover the map $\la \mapsto R_\la$ is linear.
\end{prop}

\begin{proof}
The existence follows from \eqref{e:defRla}
and subsequent arguments
(for $\la=0$, define $R_\la=0$).
Before showing uniqueness, we first prove that $R_\la+R_{-\la}=0$
for any pair of operators $R_\la$ and $R_{-\la}$ satisfying (1)--(3)
with parameters $\la$ and $-\la$, respectively.

Consider $L=R_\la+R_{-\la}$
and let $\Psi$ be the Minkowski norm on $V$ associated to~$B$.
As  $\partial_x^+ \Psi$ is a subadditive function on $X$, we have
$$
\pd^+_x \Psi(L(x)) 
= \pd^+_x \Psi\big(R_\la(x) + R_{-\la}(x)\big)
\le \pd^+_x \Psi\big(R_\la(x) + \lambda(x)\nu\big)
+ \pd^+_x \Psi\big(R_{-\la}(x) - \la(x)\nu\big)
= 0
$$
where the last identity follows from (2) and Lemma \ref{l:tangent cone}.
With this inequality and the identity $\tr L=0$ at hand,
we argue as in Lemma \ref{l:linear tangent vector field}. 
Namely, we regard $L$ as a linear vector field on $X$ and observe that,
due to the above relation $\pd^+_x \Psi(L(x))\le 0$,
$\Psi$ is non-increasing along trajectories of $L$.
Therefore $e^{tL}(K)\subset K$ for all $t\ge 0$.
On the other hand, since $\tr L=0$, the operators $e^{tL}$
preserve the volume on~$X$. 
Since $e^{tL}(K)$ is contained in $K$ and has the same volume,
we conclude that in fact $e^{tL}(K)=K$ for all $t\ge 0$.
Hence the 1-parameter group $\{e^{tL}\}$ is contained in $\I(K)$,
therefore $L\in T_e\I(K)$, and then the orthogonality (3) implies that $L=0$.

Thus we have shown that $R_\la+R_{-\la}=0$ for any choice of $R_\la$ and $R_{-\la}$.
This implies that the choice of $R_\la$ is, in fact, unique.

The same identity $R_\la+R_{-\la}=0$ shows that one can replace
``forward tangent'' in (2) by ``tangent'' (see Definition \ref{d:forward tangent}).
Upon this replacement all conditions (1)--(3) become linear in $\la$,
see Lemma~\ref{l:tangent linear}.
This and the uniqueness of $R_\la$ imply that $R_\la$ depends on $\la$ linearly.
\end{proof}

\subsection*{Proof of Proposition \ref{prop:Rexists}}
Define $R\co X^*\to\Hom(X,X)$ by $R(\la)=R_\la$
where $R_\la$ is provided by Proposition \ref{prop:Rdef}.
The requirements of Proposition \ref{prop:Rexists}
follow from the respective properties of $R_\la$ in Proposition \ref{prop:Rdef}.
Recall that the above constructions work for all $X\in\Gr_n(V)$ where
the cross-section area is differentiable, and this condition is
satisfied almost everywhere on $\Gr_n(X)$.
This finishes the proof of Proposition~\ref{prop:Rexists}.
\qed

\begin{remark}
The arguments in this section do not require \textit{all} cross-sections of $B$
to be linearly equivalent.
They work just as well if we only assume this for hyperplanes
from an arbitrarily small neighborhood of $X$ in $\Gr_n(V)$.
Thus we also have the following local version of Proposition \ref{prop:Rexists}
(compare with \cite{mono}):

{\it
Let $V$ be a vector space, $\dim V=n+1\ge 3$.
Let $B\subset V$ be a convex body with $0$ in the interior
and $\mathcal U\subset\Gr_n(V)$ a nonempty open set
such that all cross-sections of $B$
by hyperplanes from $\mathcal U$ are linearly equivalent.
Then for almost every $X\in\mathcal U$
there exist a vector $\nu\in V\setminus X$
and a linear map
$ R\co X^*\to\Hom(X,X) $
with the same properties as in Proposition \ref{prop:Rexists}.
}
\end{remark}

\section{Degenerate cases of Proposition \ref{prop:Rtangent}}\label{sec:degenerate}

Let $X$ be a vector space, $\dim X=n\ge 2$,
and let $R\co X^*\to\Hom(X,X)$ be a linear map
such that $\tr R(\la)=0$ for all $\la\in X^*$
(cf.\ Proposition~\ref{prop:Rtangent}).
As in the previous sections, we denote $R(\la)$ by $R_\la$.
To avoid repetitions of the formula $\tr R_\la=0$,
we regard $R$ as a map from $X^*$ to the space
$$
\sl(X)=\{L\in\Hom(X,X):\tr L=0\}.
$$

For  $p\in X$ we denote by $p^\bot$
the set $\{\la \in X^* \colon \lambda(p) = 0\}$ 
of covectors vanishing at~$p$.
Consider the family
$$
S^R=\{S^R_p\}_{p\in X}
$$
of linear maps $S^R_p\co p^\bot \to X$  defined by
\begin{equation}\label{e:SR}
 S^R_p(\la) = R_\la(p), \qquad \la\in p^\bot .
\end{equation}
We denote the image of $S^R_p$ by $T^R_p$, i.e.,
\begin{equation}\label{e:TR}
 T^R_p = S^R(p^\bot) = \{ R_\la(p) \mid \la\in p^\bot \} .
\end{equation}
The assumption \eqref{e:tangency assumption} of Proposition \ref{prop:Rtangent} can be
restated as follows:
for every $p\in\pd K$, 
$T^R_p$ consists of vectors
tangent to $\pd K$ at~$p$.

Clearly $T^R_p$ is a linear subspace of $X$ and $T^R_0=\{0\}$.
If $p\ne 0$, then $\dim T^R_p\le\dim p^\bot=n-1$.

\begin{defn}\label{d:degenerate point}
For a point $p\in X\setminus\{0\}$,
we say that $R$ is \textit{degenerate at $p$} if $\dim T^R_p<n-1$.
We say $R$ is \textit{degenerate} if there exists $p\in X\setminus\{0\}$
such that $R$ is degenerate at~$p$.
\end{defn}

The goal of this section is to prove Proposition \ref{prop:Rtangent}
assuming that $R$ is degenerate, see Proposition~\ref{p:all degenerate cases}.
We begin with a couple of algebraic lemmas that do not
depend on the dimension and do not involve~$K$.

\begin{lemma} \label{l:R is determined by SR}
A linear map $R\co X^*\to\sl(X)$ is uniquely determined by~$S^R$
(see \eqref{e:SR}).
In other words, if $R$ and $R'$ are linear maps from $X^*$ to $\sl(X)$
such that $S^R=S^{R'}$, then $R=R'$.
\end{lemma}

\begin{proof}
Since $S^R$ is linear in $R$, it suffices to show that $R=0$ whenever $S^R=0$.
The relation $S^R=0$ boils down to the following:
$R_\la(p) = 0$ for all $\lambda \in X^*$ and $p\in X$
such that $\lambda(p) = 0$.
Equivalently, ${\ker\la \subset \ker R_\la}$
for all $\la \in X^*$.
This inclusion implies that for every $\la\in X^*\setminus\{0\}$ 
there is a vector $v_\la \in X$ such that
$R_\la(x) = \la(x) v_\la$ for all $x\in X$.
We are going to show that $v_\la$ is independent of~$\la$.

Let $\la,\mu\in X^*\setminus\{0\}$.
Since $R_\la$ is linear in $\la$, we have
$$
0 = R_{\la+\mu}(x) - R_\la(x) - R_\mu(x)
= \la(x) (v_{\la+\mu}-v_\la) + \mu(x) (v_{\la+\mu}-v_\mu) 
$$
for all $x\in X$.
If $\la$ and $\mu$ are linearly independent then there exists $x_0\in X$
such that $\mu(x_0)=0$ and $\la(x_0)\ne 0$.
Substituting $x=x_0$ in the above identity we obtain that $v_{\la+\mu}=v_\la$.
Similarly we have $v_{\la+\mu}=v_\mu$, hence $v_\la=v_\mu$.
We have shown that $v_\la=v_\mu$ for any linearly independent $\la,\mu\in X^*$,
this implies that $v_\la$ is independent of $\la\in X^*\setminus\{0\}$.

Thus we have a fixed vector $v\in X$ such that $R_\la(x) = \la(x) v$
for all $\la\in X^*\setminus\{0\}$ and $x\in X$.
The trace of the operator $x\mapsto \la(x)v$ equals $\la(v)$,
hence $0 = \tr R_\la = \la(v)$ for all $\la\in X^*\setminus\{0\}$.
This implies that $v = 0$ and thus $R = 0$.
\end{proof}

For $\la,\mu\in X^*$ and a linear map $R\colon X^*\to\sl(X)$, define a quadratic homogeneous map
$V^R_{\la,\mu}\co X\to X$ by
\begin{equation}\label{e:VR}
V^R_{\la,\mu}(x) = \mu(x) R_\la(x) - \la(x) R_\mu(x) , \qquad x\in X.
\end{equation}
Note that $V^R_{\la,\mu}$ is bi-linear and skew-symmetric in $\la$ and $\mu$.
We regard $V^R_{\la,\mu}$ as a vector field on~$X$.

\begin{lemma} \label{l:VR}
Let $f_1,\dots,f_n$ be a basis of $X^*$ and
$V_{ij}=V^R_{f_i,f_j}$ for $i,j\in\{1,\dots,n\}$,
see \eqref{e:VR}.
Then
\begin{enumerate}
\item 
For every $p\in X$,
$$
T^R_p = \{ V^R_{\la,\mu}(p) \mid \la,\mu\in X^* \}
 = \linspan \{ V^R_{ij}(p) \mid 1 \le i<j \le n \}
$$
where $T^R_p$ is defined in \eqref{e:TR}.
\item 
$R$ is uniquely determined by the collection $\{V^R_{ij}\}_{1\le i<j\le n}$.
That is, if $R$ and $R'$ are linear maps from $X^*$ to $\sl(X)$
such that $V^R_{ij}=V^{R'}_{ij}$ for all $1\le i<j\le n$,
then $R=R'$.
In particular, if $V^R_{ij}=0$ for all $1\le i<j\le n$, then $R=0$.
\end{enumerate}
\end{lemma}

\begin{proof}
The first claim is trivial for $p=0$.
Let $p\in X\setminus\{0\}$ and $\la,\mu\in X^*$.
The linearity of~$R$ and \eqref{e:VR} imply that
$V^R_{\la,\mu}(p) = R_\alpha(p)$ where $\alpha=\mu(p)\la-\la(p)\mu$.
Clearly $\alpha(p)=0$, thus $\alpha\in p^\bot$
and therefore $R_\alpha(p) = S^R_p(\alpha)\in T^R_p$. 
This proves the inclusion 
$\{ V^R_{\la,\mu}(p) \} \subseteq T^R_p$.
To prove the converse, choose $\mu\in X^*$ such that $\mu(p)=1$
and observe that $R_\la(p)=V^R_{\la,\mu}(p)$ for all $\la\in p^\bot$.
This and \eqref{e:TR} imply that $T^R_p \subseteq \{ V^R_{\la,\mu}(p) \}$
and the first equality in (1) follows.

The second equality in (1) follows from the fact that
$V^R_{\la,\mu}$ is bi-linear and skew-symmetric in $\la$ and~$\mu$.
Indeed, this fact implies that each $V^R_{\la,\mu}$
is a linear combination of $V^R_{ij}$, $1\le i<j\le n$.

To prove (2), observe again that the collection
$\{V^R_{ij}\}_{1\le i<j\le n}$
determines the family $\{ V^R_{\la,\mu} \}_{\la,\mu\in X^*}$
by means of linear combinations.
For $p\in X\setminus\{0\}$ and $\la\in p^\bot$
pick $\mu\in X^*$ such that $\mu(p)=1$ and
conclude again that $S^R_p(\la)=R_\la(p)=V^R_{\la,\mu}(p)$.
Thus $S^R$ is determined by the collection $\{V^R_{ij}\}_{1\le i<j\le n}$.
By Lemma~\ref{l:R is determined by SR}, $S^R$ uniquely determines $R$
and the claim follows.
\end{proof}

Now let $K\subset X$ be a convex body with 0 in the interior
satisfying the assumptions of Proposition~\ref{prop:Rtangent}
with respect to~$R$.
We emphasize the following implication of Lemma \ref{l:VR}:

\begin{lemma}\label{l:VRtangent}
Let $X$, $K$ and $R$ be as in Proposition~\ref{prop:Rtangent}
and let $V^R_{\la,\mu}$ be defined by \eqref{e:VR}.
Then for all $\la,\mu\in X^*$,
the vector field $V^R_{\la,\mu}$ is tangent to~$\pd K$.
\end{lemma}

\begin{proof}
Let $p\in\pd K$.
By Lemma \ref{l:VR}(1) we have $V^R_{\la,\mu}(p)\in T^R_p$.
By \eqref{e:tangency assumption}
and the definition of $T^R_p$ (see \eqref{e:TR}),
all vectors from $T^R_p$ are tangent to $\pd K$ at $p$,
hence the result.
\end{proof}

From now on, we restrict ourselves to dimension $n=3$.

\subsection{Degeneration everywhere}
In this subsection we handle the case when $R$ is degenerate at all points.
Our goal is to prove the following.

\begin{prop}\label{p:fully degenerate}
Let $X$, $K$ and $R$ be as in Proposition~\ref{prop:Rtangent}
and assume that $R$ is degenerate at all points $p\in X\setminus\{0\}$
in the sense of Definition \ref{d:degenerate point}.
Then the conclusion of Proposition~\ref{prop:Rtangent} holds,
namely for every $\la\in X^*$ the map $R_\la$
regarded as vector field on $X$ is tangent to~$\pd K$.
\end{prop}

\begin{proof}
Since $n=3$, the degeneracy means that $\dim T^R_p\le 1$ for all $p\in X$.
Fix a basis $e_1,e_2,e_3$ of~$X$ and the corresponding
dual basis $f_1, f_2, f_3$ of $X^*$.
Define quadratic vector fields $V_1,V_2,V_3$ on $X$ by
$$
 V_1 = V^R_{23}, \qquad
 V_2 = V^R_{31}, \qquad
 V_3 = V^R_{12} 
$$
where $V^R_{ij}=V^R_{f_i,f_j}$ as in Lemma \ref{l:VR}.
For $x=\sum x_i e_i$ the definitions expand to
\begin{equation}\label{e:defVi}
\begin{aligned}
 V_1(x) &= x_3 R_{f_2}(x) - x_2 R_{f_3}(x), \\
 V_2(x) &= x_1 R_{f_3}(x) - x_3 R_{f_1}(x), \\
 V_3(x) &= x_2 R_{f_1}(x) - x_1 R_{f_2}(x).
\end{aligned}
\end{equation}
Observe that
\begin{equation}\label{e:Virelation}
x_1V_1(x) + x_2V_2(x) + x_3 V_3(x) = 0
\end{equation}
for all $x=\sum x_ie_i\in X$; this follows immediately from \eqref{e:defVi}.

By Lemma \ref{l:VRtangent}, the vector fields $V_1,V_2,V_3$ are tangent to $\pd K$.
By Lemma \ref{l:VR}(1) and the assumption that $\dim T^R_p\le 1$,
the vectors $V_i(p)$, $i=1,2,3$, are collinear for every $p\in X$.
The following algebraic lemma characterizes such triples of vector fields.

\begin{lemma}\label{l:3 collinear fields}
Let $V_1,V_2,V_3$ be quadratic homogeneous vector fields on $X\simeq\R^3$
such that
\begin{equation}\label{e:3fields dim1}
 \dim \linspan \{ V_1(x), V_2(x), V_3(x) \} \le 1 
\end{equation}
for all $x\in X$.
Then at least one of the following holds:
\begin{enumerate}
\item
There exist a quadratic vector field $Q$ on $X$ and constants
$C_1, C_2, C_3 \in \R$ such that $V_i = C_i\cdot Q$ for $i = 1, 2, 3$.
\item
There exist a linear vector field $L$ on $X$ and linear functions
$\ell_1, \ell_2, \ell_3 \in X^*$ such that
$V_i = \ell_i\cdot L$ for $i = 1, 2, 3$.
\item
There exist a vector $v \in X$ and quadratic forms
$Q_1, Q_2, Q_3 \colon X \to \R$ such that $V_i =Q_i \cdot v$ for $i = 1, 2, 3$.
\end{enumerate}
\end{lemma}

\begin{proof}
We identify $X$ with $\R^3$ and regard $V_1,V_2,V_3$ as $\R^3$-valued
polynomials in variables $x_1,x_2,x_3$ which represent coordinates in~$\R^3$.
Consider a $3 \times 3$ matrix $\mathbf V$ whose
columns are composed of coordinate components of $V_1,V_2,V_3$.
We assume that $\mathbf V\ne 0$, otherwise the lemma is trivial.

The entries of $\mathbf V$ are quadratic forms in $x_1,x_2,x_3$.
By \eqref{e:3fields dim1} all the $2\times 2$ minors of $\mathbf V$
vanish pointwise, hence they are zero as polynomials.
We now consider $\mathbf V$ as a matrix over the field $\R(x_1,x_2,x_3)$
of rational functions and conclude that its rank is no greater than~1.
This implies that $\mathbf V$ can be represented as a matrix product
\begin{equation}\label{e:rank1product}
\mathbf V = \begin{pmatrix}
    F_1\\
    F_2\\
    F_3
\end{pmatrix} \cdot 
\begin{pmatrix}
    G_1 & G_2 & G_3 
\end{pmatrix}
\end{equation}
where $F_1, F_2, F_3, G_1, G_2, G_3$ are some rational functions. Since the ring $\R[x_1,x_2,x_3]$ is a unique factorization domain and $\R(x_1, x_2, x_3)$ is its field of fractions, all $F_i$ and $G_j$ can in fact be chosen
to be polynomials rather than general rational functions.  

Indeed, represent all $F_i$ and $G_j$ as irreducible fractions of polynomials
and assume that some of them, say $F_1$, has a nontrivial denominator.
Let $q$ be a prime factor of this denominator.
Since $F_1 \cdot G^j$ is a polynomial for every $j$,
the polynomial $q$ should divide numerators of all $G^j$.
Now replace each $F_i$ by $F_iq$ and each $G_j$ by~$G_j/q$.
This operation preserves the identity \eqref{e:rank1product}
and reduces the total number of prime factors
in the denominators.
It can be applied to any of $F_i$ and $G_j$ in place of $F_1$
and repeated until all nontrivial denominators disappear.
Thus we may assume that $F_i$ and $G_j$ in \eqref{e:rank1product}
are polynomials.

Since $\mathbf V\ne 0$, at least one of the products $F_iG_j$ is nonzero.
Assume without loss of generality that $F_1G_1\ne 0$.
Since $F_1G_1$ is a homogeneous polynomial of degree~2,
the polynomials $F_1$ and $G_1$ are themselves homogeneous
and $\deg F_1 + \deg G_1 = 2$. Let $d=\deg F_1$, then $\deg G_1=2-d$.
For every nonzero $F_i$, the product $F_iG_1$ is a nonzero homogeneous
polynomial of degree~2, hence $F_i$ is homogeneous of degree $2-(2-d)=d$.
If any of $F_i$ is zero, then it is a homogeneous polynomial of any degree,
in particular of degree $d$.
Similarly, every $G_j$ is a homogeneous polynomial of degree~$2-d$.
Depending on the value of $d\in\{0,1,2\}$,
the identity \eqref{e:rank1product}
now translates to one of the possibilities (1), (2) and (3)
listed in the lemma.
\end{proof}

Now we continue the proof of Proposition \ref{p:fully degenerate}.
If $V_1=V_2=V_3=0$ then $R=0$ by Lemma \ref{l:VR}(2)
and the proposition follows trivially.
Assume that at least one of $V_1,V_2,V_3$ is nonzero
and consider the three cases from Lemma \ref{l:3 collinear fields}.

\medskip
{\bf Case 1:} $V_i = C_iQ$ where $C_i\in\R$
and $Q$ is a quadratic vector field on~$X$.
In this case the relation \eqref{e:Virelation} takes the form
$$
(C_1x_1 + C_2x_2 + C_3x_3) \cdot Q(x) = 0
$$
for all $x=\sum x_ie_i\in X$.
This identity implies that $Q=0$ or $C_1=C_2=C_3=0$.
In both cases it follows that $V_1=V_2=V_3=0$,
contrary to our assumption.

\medskip
{\bf Case 2:} $V_i = \ell_i L$ where $\ell_i\in X^*$ and $L$
is a linear vector field.
Assume without loss of generality that $V_1\ne 0$.
Then $\ell_1\ne 0$ and $L\ne 0$.
Recall that $V_1$ is tangent to $\pd K$ by Lemma \ref{l:VRtangent}.
This and Lemma \ref{l:ae tangent vector field} imply that
$L$ is tangent to~$\pd K$. Then $\tr L=0$ by Lemma \ref{l:linear tangent vector field}.

Write $\ell_i$ as $\ell_i(x) = \sum_j C_{ij} x_j$
and substitute into \eqref{e:Virelation}:
\begin{equation}\label{e:Vi relation case 2}
\left(\sum_{1 \le i, j \le 3} C_{ij} x_ix_j\right) \cdot L(x) = 0
\end{equation}
for all $x=\sum x_ie_i\in X$.
Since $L \ne 0$, the first factor in \eqref{e:Vi relation case 2} vanishes for all~$x$.
Therefore $C_{ii}=0$ and $C_{ij} + C_{ji} = 0$ for $i \ne j$.

We now claim that, for every $\la=\sum\la_if_i\in X^*$,
\begin{equation}\label{e:R in degenerate subcase 2}
R_\la = C_\la L \quad\text{where} \quad
C_\la=C_{32}\la_1 + C_{13}\la_2 + C_{21} \la_3 .
\end{equation}
To prove this, define $R'(\la)=C_\la L$ and observe that
$ \tr R'(\la) = C_\la \tr L = 0 $,
so $R'$ is a linear map from $X^*$ to $\sl(X)$.
Further, for $V^{R'}_{ij}$ defined in Lemma \ref{l:VR}
and $x=\sum x_ie_i$ we have
$$
 V^{R'}_{12}(x) = x_2 R'(f_1)(x) - x_1 R'(f_2)(x)
 = (C_{32}x_2 - C_{13}x_1) L(x) = \ell_3(x) L(x)
 = V_3(x) = V^R_{12}(x) .
$$
Here we used the identities $C_{13}=-C_{31}$ and $C_{33}=0$.
Similarly (via cyclic permutation of indices),
$V^{R'}_{23}=V^R_{23}$ and $V^{R'}_{31}=V^R_{31}$.
By Lemma \ref{l:VR} these identities imply that $R'=R$ and
\eqref{e:R in degenerate subcase 2} follows.

As $L$ is tangent to $\pd K$,
\eqref{e:R in degenerate subcase 2}
implies that $R_\la$ is tangent to $\pd K$ as well.
Thus Proposition \ref{p:fully degenerate} holds in this case.

\medskip
{\bf Case 3:} $V_i=Q_iv$ where $Q_i$ is a quadratic form
and $v$ is a fixed vector.
Assume without loss of generality that $V_1\ne 0$.
Then $Q_1\ne 0$ and $v\ne 0$.
Since $V_1$ is tangent to $\pd K$, Lemma \ref{l:ae tangent vector field}
applied to the constant vector field $v$ and a polynomial factor $Q_1$
implies that the Minkowski norm
associated to $K$ is constant along the ray $t\mapsto tv$ ($t>0$),
a contradiction.
This completes the proof of Proposition \ref{p:fully degenerate}.
\end{proof}

\subsection{Degeneration at one point}
In this subsection we consider the case when $R$
is degenerate at some point but not at all points.
Our goal is to prove the following.

\begin{prop} \label{p:partly degenerate}
Let $X$, $K$ and $R$ be as in Proposition~\ref{prop:Rtangent}.
Assume that there exist $p_0,p_1\in X\setminus\{0\}$
such that $R$ is degenerate at $p_0$ and non-degenerate at $p_1$
(see Definition \ref{d:degenerate point}).
Then $K$ is an ellipsoid.
\end{prop}

\begin{proof}
Recall that (since $n=3$) the degeneracy of $R$ at $p$
means that $\dim T^R_p\le 1$.
Consider the degenerate set
$$
 \Sigma = \{p\in X : \dim T^R_p \le 1 \}.
$$
Clearly $\Sigma$ is a cone: if $p\in\Sigma$ and $t\in\R$ 
then $tp\in\Sigma$ as well.

\begin{lemma}\label{l:X-Sigma is dense}
$X\setminus\Sigma$ is an open dense set in $X$.
\end{lemma}

\begin{proof}
Similarly to Lemma \ref{l:3 collinear fields}, introduce coordinates
in $X$ and consider a $3\times 3$ matrix $\mathbf V$
whose entries are polynomials representing
vector fields $V^R_{ij}$ from Lemma \ref{l:VR}.
A point $x\in X$ belongs to $\Sigma$ if and only if all $2\times 2$ minors
of $\mathbf V$ vanish at~$x$.
Thus $\Sigma$ is a null set of a system of polynomial equations.
Since we are assuming that $\Sigma\ne X$, it follows that
$X\setminus\Sigma$ is a dense open set.
\end{proof}

\begin{lemma} \label{l:partly degenerate one ellipse}
Let $Z\in\Gr_2(X)$ be such that $Z\setminus\{0\}$
intersects both $\Sigma$ and $X\setminus\Sigma$.
Then $K\cap Z$ is an ellipse.
\end{lemma}

\begin{proof}
Choose a basis $e_1,e_2,e_3\in X$ such that $e_1,e_2\in Z$,
$e_1\in\Sigma$ and $e_2\notin\Sigma$.
Let $f_1,f_2,f_3\in X^*$ be the corresponding dual basis.
Define a linear map
$$
 F\co Z\to X, \quad F(x) = R_{f_3}(x) \text{ for } x\in Z,
$$
and a quadratic homogeneous map
$$
 G\co Z\to X, \quad G(x) = R_{x_2f_1-x_1f_2}(x) = x_2 R_{f_1}(x)-x_1 R_{f_2}(x)
 \text{ for } x=x_1e_1+x_2e_2\in Z .
$$
Then construct a linear combination $W(x)$ of $F(x)$ and~$G(x)$
killing the third coordinate:
$$
 W(x) = f_3(G(x))\cdot F(x) - f_3(F(x))\cdot G(x).
$$
Since $W(x)\in Z$ for all $x\in Z$, we may regard $W$ as
a vector field on~$Z$.
Note that $W$ is homogeneous polynomial of degree~3.

Observe that for every $x=x_1e_1+x_2e_2\in Z$ the covectors
$f_3$ and $x_2f_1-x_1f_2$ from the definitions of $F(x)$ and $G(x)$
form a basis of $x^\perp\subset X^*$.
Therefore $\linspan\{F(x),G(x)\}=T^R_x$
(see \eqref{e:TR}).
In particular both $F(x)$ and $G(x)$ are tangent to $\pd K$
if $x\in\pd K\cap Z$.
This and Lemma \ref{l:tangent linear} imply that
$W(x)$ is tangent to $\pd K$ and hence to $\pd(K\cap Z)$.

Since $e_1 \in \Sigma$, we have $\dim T^R_{e_1}\le 1$,
therefore $F(e_1)$ and $G(e_1)$ are collinear, hence $W(e_1)=0$.
Since 
$e_2 \notin \Sigma$, we have $\dim T^R_{e_2}=2$, 
so $F(e_2)$ and $G(e_2)$ are linearly independent.
This implies that $W(e_2)\ne 0$.
Indeed, if $W(e_2)=0$ then $F(e_2),G(e_2) \in \ker f_3 = Z$,
hence $F(e_2)$ and $G(e_2)$ are tangent to $\pd(K\cap Z)$
and therefore linearly dependent, a contradiction.
Thus $W$ is a nonzero cubic vector field on $Z\simeq\R^2$,
$W$ is tangent to $\pd(K\cap Z)$ and it vanishes at~$e_1$.
By Lemma \ref{l:cubic vector field in the plane},
these properties imply that $K\cap Z$ is an ellipse.
\end{proof}

Fix $p_0\in\pd K\cap\Sigma$, such a point exists since $\Sigma$ is a cone
and $\Sigma\ne\{0\}$.
Consider the set (homeomorphic to the circle)
of all planes from $\Gr_2(X)$ that contain~$p_0$.
Lemma \ref{l:X-Sigma is dense} imply that this set contains
a dense subset of planes intersecting $X\setminus\Sigma$.
By Lemma \ref{l:partly degenerate one ellipse}, the intersection
of every such a plane with $K$ is an ellipse.
By continuity of intersection it follows that \textit{all}
cross-sections of $K$ by planes containing 0 and $p_0$
are ellipses.

Consider the tangent lines to these elliptic cross-sections at $p_0$.
These lines are tangent to $\pd K$ at $p_0$,
therefore (cf.\ Lemma \ref{l:tangent linear})
they form a tangent plane to $\pd K$ at~$p_0$.
Let $H$ be the plane through~0 parallel to this tangent plane.
Choose a basis $e_1,e_2,e_3$ in $X$ such that $e_3=p_0$
and $e_1,e_2\in H$, and let $f_1,f_2,f_3\in X^*$ be the dual basis.
We equip $X$ with the Euclidean structure such that $e_1,e_2,e_3$
is an orthonormal basis.
We call vectors from $H$ \textit{horizontal}
and vectors collinear to $p_0$ \textit{vertical}.
A plane $Z\in\Gr_2(X)$ is called \textit{vertical}
if it contains $p_0$.

For every vertical plane $Z$ the cross-section $\pd K\cap Z$
is an ellipse containing $p_0$ and having a horizontal tangent line at~$p_0$.
Such an ellipse has axes on the vertical and horizontal lines in $Z$
and it is uniquely determined by its horizontal axis $Z\cap H\cap K$.
Thus $K$ is uniquely determined by the point $p_0$ and the 
horizontal cross-section $K\cap H$.
We are going to show that $K\cap H$ is an ellipse,
this will imply that $K$ is an ellipsoid.

Observe that for every vertical plane $Z$, the ellipse $\pd K\cap Z$
has a vertical tangent line at the point of intersection with~$H$.
Thus the vertical vector $e_3=p_0$ is tangent to $\pd K$
at every point from $\pd K\cap H$.
Consider a linear vector field $L$ on $H$ defined by
$$
 L(x) = \pr_H(R_{f_3}(x)), \qquad x\in H
$$
where $\pr_H$ is the coordinate projection to~$H$.
For every $x\in\pd K\cap H$, both $R_{f_3}(x)$ and $e_3$ are
tangent to $\pd K$ at~$x$, hence
$L(x)$ is also tangent to $\pd K$ at~$x$ by Lemma \ref{l:tangent linear}.
Thus $L$ is a linear vector field on $H$ tangent to $\pd(K\cap H)$.
If $L\ne 0$ then Lemma \ref{l:cubic vector field in the plane}
applied to $K\cap H$ implies that $K\cap H$ is an ellipse.

If $L=0$, consider the third coordinate $f_3(R_{f_3}(x))$ 
as a function of $x\in H$.
Since this function is linear, it vanishes at some point
$q\in\pd K\cap H$. For this point we have $L(q)=\pr_H(R_{f_3}(q))=0$
and  $f_3(R_{f_3}(q))=0$, hence $R_{f_3}(q)=0$.
Since $f_3\in q^\perp$, this implies that $R$ is degenerate at $q$.
Now we can repeat the above arguments for $q$ in place of $p_0$
and conclude that all cross-sections of $K$ by planes
containing 0 and $q$ are ellipses.
In particular this applies to the plane~$H$.

Thus we have shown that $K\cap H$ is an ellipse.
Since $K$ is uniquely determined by $K\cap H$ and $p_0$ as explained above,
$K$ coincides with the ellipsoid having a vertical axis
and a given elliptic horizontal cross-section.
This finishes the proof of Proposition~\ref{p:partly degenerate}.
\end{proof}

Now we combine Proposition \ref{p:fully degenerate}
and Proposition \ref{p:partly degenerate}
to obtain  the main result of this section:

\begin{prop}\label{p:all degenerate cases}
Proposition~\ref{prop:Rtangent} holds true if $R$ is degenerate
in the sense of Definition \ref{d:degenerate point}.
\end{prop}

\begin{proof}
If $R$ is degenerate at all points then the result follows
from Proposition \ref{p:fully degenerate},
otherwise it follows from Proposition \ref{p:partly degenerate}
and Lemma \ref{l:Rtangent for ellipsoid}.
\end{proof}

\section{Non-degenerate case}\label{sec:non-degenerate}

In this section we prove Proposition \ref{prop:Rtangent}
in the case when $R$ is non-degenerate in the sense
of Definition \ref{d:degenerate point}, that is,
\begin{equation}\label{e:nondegeneracy assumption}
\dim T_p^R = n-1= 2 \quad\text{for all $p\in X\setminus\{0\}$},
\end{equation}
where $T_p^R$ is defined by \eqref{e:TR}.
We fix $X$, $K$ and $R$ from Proposition \ref{prop:Rtangent}
and assume \eqref{e:nondegeneracy assumption} for the rest of this section.

Let $\Psi$ be the Minkowski norm on $X$ associated to $K$.
By Lemma \ref{l:tangent cone} and the homogeneity of $\Psi$ and~$R$,
the assumption \eqref{e:tangency assumption} of Proposition \ref{prop:Rtangent} implies that
\begin{equation}\label{e:tangency assumption in terms of Psi}
 \pd^+_x \Psi(\pm R_\la(x))=0 
 \quad\text{for all $\la\in X^*$ and $x\in\ker\la$}.
\end{equation}

\begin{lemma}\label{l:symmetric and analytic}
$K$ is symmetric with respect to~0
and $\pd K$ is a real analytic surface.
\end{lemma}

\begin{proof}
Recall that for every $p\in\pd K$
all vectors from $T_p^R$ are tangent to $\pd K$ at~$p$
(cf.\ \eqref{e:tangency assumption} and \eqref{e:TR}).
Since $\dim T_p^R=2$, this means that $T_p^R$ is the tangent plane
of $\pd K$ at~$p$.
Thus $\pd K$ is an integral surface of the two-dimensional
distribution $\{T_p^R\}_{p\in X\setminus\{0\}}$ on $X\setminus\{0\}$.
Since this distribution is generated by algebraic vector fields
(see Lemma \ref{e:VR}), it is real analytic, hence so is~$\pd K$.

To prove the symmetry of $K$, consider the intersection of $\pd K$
with an arbitrary straight line passing through~0.
It is a pair of points $p_1,p_2\in \pd K$ such that $p_2=cp_1$ where $c<0$.
Define a function $F\co X\to\R$ by $F(p) = \Psi(cp)$ for all $p\in X$.
For every $p\in\pd K$, the differential
$d_pF = c\cdot d_{cp}\Psi$ vanishes on~$T_p^R$.
Indeed, consider $p'=\Psi(cp)^{-1}cp\in\pd K$
and observe that $T_{p'}^R=T_p^R$
since the distribution $\{T_p^R\}$ is generated by homogeneous vector fields.
Since $T_{p'}^R$ is tangent to $\pd K$ at~$p'$, it follows that
$d_{p'}\Psi$ vanishes on~$T_{p'}^R=T_p^R$.
Hence, by the homogeneity of $\Psi$,
$d_{cp}\Psi$ and hence $d_pF$ vanish on~$T_p^R$.

Since $\pd K$ is connected and $d_pF$ vanishes
on its tangent plane $T_p^R$ for every $p\in\pd K$,
the restriction $F|_{\pd K}$ is constant.
Substituting $p_1$ one sees that this constant is $F(p_1)=\Psi(p_2)=1$.
On the other hand, it equals $F(p_2)=\Psi(c^2p_1)=c^2$.
Therefore $c^2=1$, hence $c=-1$ and $p_2=-p_1$.
Thus for every $p_1\in\pd K$ we have $-p_1\in\pd K$,
therefore $K$ is symmetric.
\end{proof}

We reuse some notation from \S\ref{subsec:nu}
but here the constructions are applied to~$X$ in place of~$V$.
For linearly independent vectors $u,v\in X$ define a plane
$\Pi_{u\wedge v}\in\Gr_2(X)$ by 
$$
 \Pi_{u\wedge v} = \linspan(u,v).
$$
As in \eqref{e:affine area},
define a positively 1-homogeneous area function
$A\co \Lambda^2X \to \R_+$ for $K$ by $A(0) = 0$ and 
\begin{equation}\label{e:area 2D}
A(u \wedge v) = \vol_{u \wedge v}(K\cap \Pi_{u\wedge v})^{-1} \text{ for linearly independent } u, v \in X
\end{equation}
where $\vol_{u \wedge v}$ is the Haar measure on $\Pi_{u\wedge v}$
normalized in such a way that the parallelogram with edges $u,v$ has measure~$1$.

Let $IK \subset \Lambda^2 X$ be the intersection body of $K$, that is, 
\begin{equation}\label{e:intersection body}
  IK = \left\{\sigma \in \Lambda^2 X \mid A(\sigma) \le 1 \right\} .
\end{equation}
We place the intersection body in $\Lambda^2X$ rather than in $X$ or $X^*$
(cf.~\cites{Bus49,Thompson}) to ensure its independence of additional structures
(such as an inner product) on~$X$.
By a classical result of Busemann \cite{Bus49},
the symmetry of $K$ (see Lemma \ref{l:symmetric and analytic})
implies that $IK$ is convex.
Note that $A$ from \eqref{e:area 2D} is the Minkowski norm
associated to~$IK$.

The second claim of Lemma \ref{l:symmetric and analytic} implies
that $\pd IK$ is a real analytic surface in~$\Lambda^2 X$.
Note that the analyticity implies that $IK$ is strictly convex in the sense that
its boundary does not contain straight line segments.
Indeed, if $\pd IK$ contains a segment of a straight line then by analyticity
it contains the entire line, contrary to the fact that it is the boundary of a compact set.
See also \cite{Bus49}*{IIa} for general comments about strict convexity
of intersection bodies.

The next lemma is the key construction of the proof.

\begin{lemma}\label{l:flow of planes}
There exist a quadratic homogeneous vector field $W$ on $\Lambda^2 X$ such that
\begin{enumerate}
\item $W$ is tangent to $\pd IK$.
\item For every trajectory $\ga$ of $W$ in $\pd IK$,
all cross-sections $K\cap\Pi_{\ga(t)}$, $t\in\R$, are linearly equivalent.
\item If $\sigma\in\pd IK$ is such that $W(\sigma)=0$ then
$K\cap\Pi_{\sigma}$ is an ellipse.
\end{enumerate}
\end{lemma}

\begin{proof}
Fix a linear isomorphism
$$
\varphi\co \Lambda^2 X \to X^*
$$
such that $u,v\in\ker\varphi(u\wedge v)$ for all $u,v\in X$.
Such $\varphi$ can be constructed, for instance, by choosing a
volume form $\omega$ on $X$ and setting $\varphi(u\wedge v)=\omega(u,v,\cdot)$.

We define the desired vector field $W\co \Lambda^2 X\to\Lambda^2 X$ by
\begin{equation}\label{e:plane flow def}
 W(\sigma)  = R_{\varphi(\sigma)}(u) \wedge v + u \wedge R_{\varphi(\sigma)}(v)
 \quad \text{if $\sigma = u\wedge v$ where $u,v\in X$}.
\end{equation}
The definition does not depend on the choice of $u$ and~$v$ representing $\sigma$.
Indeed, if $\sigma\ne 0$, any other pair of vectors representing $\sigma$
can be obtained from the original one by a series of elementary
transformations of the form $(u,v)\mapsto (v,-u)$,
$(u,v)\mapsto (cu,c^{-1}v)$ and $(u,v)\mapsto (u,v+cu)$
where $c\in\R\setminus\{0\}$.
Applying any of these transformations to $u$ and $v$ in \eqref{e:plane flow def}
and using the linearity of $R_{\varphi(\sigma)}$ and the bi-linearity
and skew-symmetry of $\wedge$, one sees that the result does not change.

First we check that $W$ is indeed a quadratic vector field.
Choose a basis $e_1,e_2,e_3$ in~$X$ and introduce coordinates on $\Lambda^2 X$
by means of the basis $e_2\wedge e_3$, $e_3\wedge e_1$, $e_1\wedge e_2$.
Consider $\sigma\in\Lambda^2 X$ and let $x,y,z$ be the coordinates of $\sigma$:
$$
 \sigma = x e_2\wedge e_3 + y e_3\wedge e_1 + z e_1\wedge e_2 .
$$
Assuming $z\ne 0$, we have $\sigma=\frac1z u\wedge v$
where $u=ze_1-xe_3$ and $v=ze_2-ye_3$.
Then, by \eqref{e:plane flow def},
$$
 W(\sigma) = \frac1z\, R_{\varphi(\sigma)}(u) \wedge v 
 + \frac1z\, u \wedge R_{\varphi(\sigma)}(v) 
$$
or, equivalently
\begin{equation}\label{e:W(x,y,z)}
 z\cdot W(\sigma) = R_{\varphi(\sigma)}(u) \wedge v 
 + u \wedge R_{\varphi(\sigma)}(v) .
\end{equation}
Substituting the coordinate expressions for $\sigma$, $u$ and $v$
we regard both sides of \eqref{e:W(x,y,z)} as functions
of real variables $x,y,z$.
The identity \eqref{e:W(x,y,z)} holds for for all $x,y,z$ such that $z\ne 0$
and hence, by continuity, for all $x,y,z\in\R$.
Since $\sigma$, $u$ and $v$ are linear functions of $(x,y,z)$,
the right-hand side of \eqref{e:W(x,y,z)}
is a homogeneous polynomial of degree~3 (with values in $\Lambda^2 X$).
The left-hand side shows that this polynomial vanishes at $z=0$,
therefore it is divisible by $z$ as a polynomial.
Thus $W$ is a quadratic vector field on $\Lambda^2 X$.

Now we verify the properties (1)--(3).
Fix $\sigma\in \Lambda^2 X\setminus\{0\}$ and choose $u_0,v_0\in X$
such that $\sigma=u_0\wedge v_0$.
Let $u=u(t)$ and $v=v(t)$ be the solutions of the o.d.e.\ system
\begin{equation}\label{e:uv ode}
\begin{cases}
 \dot u = R_{\varphi(u\wedge v)}(u) \\
 \dot v = R_{\varphi(u\wedge v)}(v) 
\end{cases}
\end{equation}
with initial values $u(0)=u_0$ and $v(0)=v_0$,
and let $J \ni 0$ be the largest interval where the solution
is defined. (We will show later that $J = \R$).
Let $\ga(t)=u(t)\wedge v(t)$, then
$$
 \dot\ga(t) = \dot u(t)\wedge v(t) + u(t) \wedge \dot v(t) = W(\ga(t))
$$
by \eqref{e:uv ode} and \eqref{e:plane flow def}.
Thus $\ga$ is a trajectory of~$W$ with $\ga(0)=\sigma$,
in particular $\ga(t)\ne 0$ for all~$t \in J$.

For each $t \in J$, let $L_t\co\R^2\to X$ be the linear map such that
$L_t(e_1)=u(t)$ and $L_t(e_2)=v(t)$
where $e_1$ and $e_2$ is the standard basis of~$\R^2$.
Fix $w=c_1e_1+c_2e_2\in\R^2$ and consider a curve $\alpha=\alpha(t)$ in~$X$
defined by 
$$
\alpha(t)=L_t(w)=c_1u(t)+c_2v(t) .
$$
By \eqref{e:uv ode} and the linearity of $R_{\ga(t)}$,
the velocity of $\alpha$ is given by
\begin{equation}\label{e:dot alpha}
 \dot\alpha(t) = R_{\varphi(\ga(t))}(\alpha(t)) .
\end{equation}
Recall that $\varphi(\ga(t))=\varphi(u(t)\wedge v(t))$ is a co-vector
whose kernel contains $u(t)$ and $v(t)$, hence it also contains
their linear combination $\alpha(t)$.
This fact, \eqref{e:dot alpha} and \eqref{e:tangency assumption in terms of Psi}
imply that $\frac d{dt}\Psi(\alpha(t))=0$,
hence the function $t\mapsto\Psi(\alpha(t))$ is constant.
Since $\alpha(t)=L_t(w)$ where $w$ is an arbitrary vector from $\R^2$,
we conclude that $\Psi\circ L_t$ is independent of~$t$:
\begin{equation}\label{e:Psi Lt constant}
 \Psi\circ L_t = \Psi\circ L_0 \quad\text{for all } t\in J.
\end{equation}
Note that \eqref{e:Psi Lt constant} implies that $u(t)$ and $v(t)$
are confined in a compact subset of~$X$.
Therefore $u(t)$ and $v(t)$ are defined by \eqref{e:uv ode} for all $t\in\R$,
thus $J=\R$.

The relation \eqref{e:Psi Lt constant} implies that the pre-image
$L_t^{-1}(K)\subset\R^2$ is independent of~$t$.
Observe that
$$
 \Pi_{\ga(t)}=\linspan\{u(t),v(t)\}=L_t(\R^2) ,
$$
hence $K\cap\Pi_{\ga(t)}$ is linearly equivalent to
$L_t^{-1}(K)=L_0^{-1}(K)$ and (2) follows.

To prove (1), observe that for $A$ defined by \eqref{e:area 2D}
we have $A(\ga(t)) = |L_t^{-1}(K)|^{-1}$
where $|\cdot|$ is the standard Euclidean area in~$\R^2$.
Since $L_t^{-1}(K)$ is independent of~$t$, so is $A(\ga(t))$.
Since $\ga$ is an arbitrary trajectory of~$W$, it follows that
$W$ is tangent to the level sets of $A$ and
in particular to~$\pd IK$.

It remains to prove (3). Assume that $W(\sigma)=0$,
then $\ga(t)=\sigma$ and hence $\Pi_{\ga(t)}=\Pi_\sigma$
for all~$t$.
Now we may regard $L_t$ as a (bijective) linear map
from $\R^2$ to the fixed plane $\Pi_\sigma$, 
and then \eqref{e:Psi Lt constant} implies
that $\{L_t\circ L_0^{-1}\}_{t\in\R}$ is a one-parameter family
of linear self-equivalences of $K\cap\Pi_\sigma$.
If this family is not constant then the group
$\I(K\cap\Pi_\sigma)$ of self-equivalences is not discrete
and hence $K\cap\Pi_\sigma$ is an ellipse, as claimed.
Otherwise $\{L_t\}$ and hence $u(t)$ are constants.
This and \eqref{e:uv ode} imply that $R_{\varphi(\sigma)}(u_0)=0$,
so the map $\la\mapsto R_\la(u_0)$ is not injective on $u_0^\perp$.
This contradicts our standing assumption \eqref{e:nondegeneracy assumption}
for $p=u_0$.
This proves (3) and finishes the proof of Lemma \ref{l:flow of planes}.
\end{proof}

Now we finish the proof of Proposition \ref{prop:Rtangent}.

\subsection*{Proof of Proposition \ref{prop:Rtangent}.}
Let $X,K,R$ be as in Proposition \ref{prop:Rtangent}. The case when
$R$ is degenerate is covered by Proposition~\ref{p:all degenerate cases},
so we assume the non-degeneracy condition \eqref{e:nondegeneracy assumption}.
Due to Lemma~\ref{l:Rtangent for ellipsoid} it suffices to show
that $K$ is an ellipsoid.

By Lemma \ref{l:symmetric and analytic}, $K$ is symmetric
and $\pd K$ is a real analytic surface.
Let $IK\subset\Lambda^2X$ be the intersection body of $K$,
see \eqref{e:intersection body},
$W$ a quadratic vector field on $\Lambda^2X$ constructed
in Lemma~\ref{l:flow of planes},
and $S=\pd IK$.
Define
$$
 \mathcal E = \{ \sigma\in S :
 K \cap \Pi_\sigma \text{ is an ellipse} \}
$$
(recall that $\Pi_\sigma$ is the plane associated to a bivector $\sigma$).
Clearly $\mathcal E$ is a closed set.
Lemma~\ref{l:flow of planes}(3) implies that $\mathcal E$ 
contains all zeroes of $W$ on~$S$.
If $W=0$ then $\mathcal E=S$, which means that all central cross-sections
of $K$ are ellipses and hence $K$ is an ellipsoid.
We assume that $W\ne 0$ and proceed in several steps.

\smallskip
\textit{Step 1}:
We show that $\mathcal E$ contains a non-constant orbit
$\{\ga(t)\}$ of~$W$.
As explained before Lemma~\ref{l:flow of planes},
$S$ is a smooth strictly convex surface.
Therefore we can apply Lemma \ref{l:fixedpts} to~$S$
and obtain 
a non-constant orbit $\{\ga(t)\}\subset S$ of $W$
whose closure contains a zero of~$W$.
Lemma \ref{l:flow of planes}(2) implies that
all cross-sections $K\cap \Pi_{\ga(t)}$ are linearly equivalent.
By continuity they are also linearly equivalent to $K\cap\Pi_{\sigma}$,
which is an ellipse by Lemma~\ref{l:flow of planes}(3).
Hence $\ga(t)\in\mathcal E$ for all $t\in\R$.

\smallskip
\textit{Step 2}:
We show that $\mathcal E$ contains three linearly independent elements.
Suppose to the contrary that $\mathcal E$ is contained in some plane
$H\in\Gr_2(\Lambda^2X)$ and hence in the topological circle $S \cap H$.
In particular the orbit from Step~1 is contained in $H$,
therefore $W(\ga(t))\in H$ for all~$t$.
Since $W$ is a quadratic homogeneous polynomial,
it follows that $W(q)\in H$ for all $q\in H$.
Hence every trajectory of $W$ starting at a point on the circle
$S \cap H$ stays on this circle forever and therefore converges
to a fixed point of~$W$.
As explained above, this implies that the entire trajectory
is contained in~$\mathcal E$.
Thus all points of $S \cap H$ belong to $\mathcal E$,
so in fact $\mathcal E=S \cap H$.

Pick $q\in S\setminus H$, consider the orbit of $W$ starting at~$q$,
and apply the Poincar\'e-Bendixson Theorem similarly to the proof
of Lemma \ref{l:closed orbit}.
If the closure of the orbit contains a fixed point
then we have $q\in\mathcal E$ similarly to Step~1.
Otherwise the closure contains a closed orbit of $W$,
this closed orbit bounds a disc $D\subset S\setminus H$,
which contains a fixed point  by Brouwer's fixed-point theorem,
and this fixed point belongs to~$\mathcal E$
by Lemma~\ref{l:flow of planes}(3).
In both cases we have found an element of $\mathcal E$ outside~$H$.
This contradicts our assumptions and thus proves the claim of Step~2.

\smallskip
\textit{Step 3}:
We show that there exists a quadratic form $Q$ on $K$ such that
for every $\sigma\in\mathcal E$ the cross-section $K\cap\Pi_\sigma$
is compatible with $Q$ in the sense that $Q=1$ on $\pd K\cap\Pi_\sigma$.

Fix linearly independent $\eta_1,\eta_2,\eta_3\in\mathcal E$
provided by Step~2.
The linear independence 
implies that $\Pi_{\eta_1}\cap\Pi_{\eta_2}\cap\Pi_{\eta_3}=\{0\}$.
There exists a quadratic form $Q$ on $X$ compatible with the
three cross-sections $K\cap\Pi_{\eta_i}$ in the above sense.
Indeed, let $e_1,e_2,e_3\in\pd K$ be a basis of $X$ such that
the planes $\Pi_{\eta_i}$, $i=1,2,3$, are the coordinate planes
with respect to this basis: $e_i\in\Pi_{\eta_{i+1}}\cap \Pi_{\eta_{i+2}}$
where the indices are taken modulo~3.
The matrix of the desired quadratic form $Q$ in this basis
can be obtained as follows: the diagonal entries are all equal to~1
and then each off-diagonal entry is uniquely determined by the
compatibility condition on the respective coordinate plane.

Let $\sigma\in\mathcal E$ and let $Q_\sigma$ be the quadratic form
on the plane $\Pi_\sigma$ corresponding to the ellipse $K\cap\Pi_\sigma$.
Our goal is to show that $Q_\sigma = Q|_{\Pi_\sigma}$.
In the case when $\Pi_\sigma$ does not contain any basis vector,
this follows from the fact that a quadratic form on the plane
is uniquely determined by its values on the three lines
$\Pi_\sigma\cap\Pi_{\eta_i}$.
If $\Pi_\sigma$ contains a basis vector, say $e_1\in\Pi_\sigma$,
then $Q_\sigma$ is uniquely determined by its values on the two lines,
$(e_1)$ and $\Pi_\sigma\cap\Pi_{\eta_1}$,
and the derivative of $Q_\sigma$ at~$e_1$.
The derivative of $Q_\sigma$ at $e_1$ is uniquely determined from
the fact that the tangent line of the ellipse $\pd K\cap\Pi_\sigma$
at $e_1$ belongs to the plane spanned by the tangent lines
to $\pd K\cap\Pi_{\eta_2}$ and $\pd K\cap\Pi_{\eta_3}$ at $e_1$.
In both cases $Q_\sigma = Q|_{\Pi_\sigma}$ 
and therefore $K\cap\Pi_\sigma$ is compatible with $Q$.

\smallskip
\textit{Step 4}:
We show that $\Psi^2=Q$ where $Q$ is the quadratic form from Step~3
(recall that $\Psi$ is the Minkowski norm associated to $K$).
By Step~1 the set $\mathcal E$ contains a non-constant smooth curve $\{\ga(t)\}$,
and by Step~3 the equality $\Psi^2=Q$ holds on each of the planes $\Pi_{\ga(t)}$.
The union of these planes has a nonempty interior,
thus $\Psi^2=Q$ on a nonempty open subset of~$X$.
Since $\pd K$ is a real analytic surface, $\Psi$ is real analytic
on $X\setminus\{0\}$ and we conclude that $\Psi^2=Q$ everywhere.

\smallskip
The identity $\Psi^2=Q$ implies that $Q$ is positive definite and $K$ is an ellipsoid.
An application of Lemma~\ref{l:Rtangent for ellipsoid} finishes
the proof of Proposition \ref{prop:Rtangent}.
\qed

\medskip

As shown in \S\ref{subsec:proof main}, Theorem \ref{t:main} follows
from Proposition \ref{prop:Rexists} and Proposition \ref{prop:Rtangent}.

\end{document}